\documentclass[preprint]{imsart}
\RequirePackage{amsthm,amsmath}
\RequirePackage{color}
\RequirePackage[numbers]{natbib}
\usepackage{amssymb,amsfonts,bbm}
\usepackage{comment}

\arxiv{arXiv:0000.0000}

\startlocaldefs
\numberwithin{equation}{section}
\theoremstyle{plain}

\newtheorem{dummytheorem}{Dummy-Theorem}[section]
\newcommand{\proofendsign}{$\Box$} 

\newtheorem{lemma}[dummytheorem]{Lemma}
\newtheorem{theorem}[dummytheorem]{Theorem}
\newtheorem{proposition}[dummytheorem]{Proposition}

\newtheorem{corollary}[dummytheorem]{Corollary}

\newtheorem{example}[dummytheorem]{Example}

\newtheorem{remark}[dummytheorem]{Remark}

\renewenvironment{proof}{{\noindent \bf Proof }}
 {{\hspace*{\fill}\proofendsign\par\bigskip}}

\endlocaldefs

\newcommand{\N}{\mathbb{N}}

\newcommand{\Q}{\mathrm{Q}}
\newcommand{\R}{\mathbb{R}}

\newcommand{\E}{\mathbb{E}}

\newcommand{\pr}{\mathrm{P}}
\newcommand{\qr}{\mathrm{Q}}
\newcommand{\ex}{\mathbb{E}}

\newcommand{\eins}{\mathsf{1}}

\newcommand{\cA}{\mathcal A}
\newcommand{\cB}{\mathcal B}
\newcommand{\cC}{\mathcal C}
\newcommand{\cF}{\mathcal F}
\newcommand{\cG}{\mathcal G}

\newcommand{\cE}{\mathcal E}
\newcommand{\cK}{\mathcal K}
\newcommand{\cL}{\mathcal L}
\newcommand{\cM}{\mathcal M}
\newcommand{\cP}{\mathcal P}
\newcommand{\cQ}{\mathcal Q}

\newcommand{\cT}{\mathcal T}
\newcommand{\cTr}{\mathcal T^{r}}

\newcommand{\cX}{\mathcal X}
\newcommand{\cZ}{\mathcal Z}

\newcommand{\hZ}{\widehat{Z}}

\newcommand{\tZ}{\widetilde{Z}}

\newcommand{\argmin}{{\rm argmin}}

\newcommand{\esssup}{\mathop{\mathrm{ess\,sup}}\displaylimits}
\newcommand{\essinf}{\mathop{\mathrm{ess\,inf}}\displaylimits}

\newcommand{\OFP}{(\Omega,{\cal F},\pr)}

\newcommand{\OFPk}{(\Omega,{\cal F}_{t_{k}},\pr_{|{\cal F}_{t_{k}}})}
\newcommand{\OFFP}{(\Omega,{\cal F}, ({\cal F}_{t})_{0\leq t\leq T},\pr)}
\newcommand{\tOFFP}{(\Omega,{\cal F}, (\widetilde{{\cal F}}_{t})_{0\leq t\leq \infty},\pr)}
\newcommand{\oOmega}{\overline{\Omega}}
\newcommand{\ocF}{\overline{{\mathcal F}}}
\newcommand{\oP}{\overline{\pr}}
\newcommand{\ocPinfty}{\overline{{\mathcal P}}^{\infty}}
\newcommand{\taur}{\tau^{r}}

\def\bcswitch{\left\{\renewcommand{\arraystretch}{1.2}\begin{array}{c@{,~}c}}
\def\ecswitch{\end{array}\right.}

\newcommand{\Timesim}{\raisebox{-2.2mm}{
\unitlength 1.00pt
\linethickness{0.5pt}
\begin{picture}(10.00,20.00)
\multiput(2.00,6.00)(0.12,0.18){50}{\line(0,1){0.18}}
\multiput(2.00,15.00)(0.12,-0.18){50}{\line(0,-1){0.18}}
\put(5.00,2.00){\makebox(0,0)[cc]{$\scriptstyle i=1$}}
\put(5.00,18.00){\makebox(0,0)[cc]{$\scriptstyle m$}}
\end{picture}
}}
\newcommand{\Timesikm}{\raisebox{-2.2mm}{
\unitlength 1.00pt
\linethickness{0.5pt}
\begin{picture}(10.00,20.00)
\multiput(2.00,6.00)(0.12,0.18){50}{\line(0,1){0.18}}
\multiput(2.00,15.00)(0.12,-0.18){50}{\line(0,-1){0.18}}
\put(5.00,2.00){\makebox(0,0)[cc]{$\scriptstyle i=k$}}
\put(5.00,18.00){\makebox(0,0)[cc]{$\scriptstyle m$}}
\end{picture}
}}

\newcommand{\Timesk}{\raisebox{-2.2mm}{
\unitlength 1.00pt
\linethickness{0.5pt}
\begin{picture}(10.00,20.00)
\multiput(2.00,6.00)(0.12,0.18){50}{\line(0,1){0.18}}
\multiput(2.00,15.00)(0.12,-0.18){50}{\line(0,-1){0.18}}
\put(5.00,2.00){\makebox(0,0)[cc]{$\scriptstyle k=1$}}
\put(5.00,18.00){\makebox(0,0)[cc]{$\scriptstyle k_{j}$}}
\end{picture}
}}

\begin{document}

\begin{frontmatter}
\title{Optimal stopping under probability distortions and law invariant coherent risk measures }
\runtitle{Optimal stopping under probability distortions}

\begin{aug}
\author{\snm{Denis Belomestny}
\ead[label=e1]{denis.belomestny@uni-due.de}}
\and
\author{\snm{Volker Kr\"atschmer}\ead[label=e2]{volker.kraetschmer@uni-due.de}}


\runauthor{D. Belomestny and V. Kr\"atschmer}

\affiliation{Duisburg-Essen University}

\address{
Duisburg-Essen University\\
Faculty of Mathematics\\
Thea-Leymann-Str. 9\\
D-45127 Essen\\
Germany\\
\printead{e1}\\
\phantom{E-mail: \ }\printead*{e2}}
\end{aug}

\begin{abstract}
In this paper we study  optimal stopping problems with respect to distorted expectations of the form
\begin{eqnarray*}
\cE(X)=\int_{-\infty}^{\infty} x\,dG(F_X(x)),
\end{eqnarray*}
where \(F_X\) is the distribution function of \(X\) and \(G\) is a convex distribution function on \([0,1].\)  As a matter of fact, except for $G$ being the identity on $[0,1],$ dynamic versions of \(\cE(X)\) do not have the so-called time-consistency property necessary for the dynamic programming approach. So the standard approaches are not applicable to optimal stopping  under \(\cE(X).\) In this paper, we prove a novel representation, which relates the solution of an optimal stopping problem under distorted expectation to the sequence of standard optimal stopping problems and hence makes the  application of the standard dynamic programming-based approaches possible. 
Furthermore, by means of the well known Kusuoka representation, we extend our results to  optimal stopping under general law invariant coherent risk measures. Finally, based on our novel representations, we develop several Monte Carlo approximation algorithms and illustrate their power for  optimal stopping under Average Value at Risk and the absolute semideviation risk measures.
\end{abstract}



\end{frontmatter}

\section{Introduction.}\label{intro} 
Consider a random variable \(X\) on some atomless probability space $\OFP$ with distribution function \(F_X.\)
Let \(G\) be a fixed distribution function defined on \([0,1],\) such mappings are also known as distortion functions. Denote by \(\cE\) the expectation of \(X\) taken with respect to the distorted distribution function \(G(F_X(x)),\) i.e., 
\begin{eqnarray*}
\cE(X)=\int_{-\infty}^{\infty} x\,dG(F_X(x)).
\end{eqnarray*}
After splitting at zero and integrating by parts we may write it  in the form
\begin{eqnarray}
\label{distorted_expectation} \nonumber
\cE^g(X)&=&  
\int_{0}^{\infty}[1 - G(F_{X}(x))]~dx-\int_{-\infty}^{0}G(F_{X}(x))~dx
\\ 
&=& \int_{0}^{\infty}g(1 - F_{X}(x))~dx-\int_{-\infty}^{0}[1-g(1-F_{X}(x))]~dx,
\end{eqnarray}
where \(g\) is related to \(G\) via \(G(x)=1-g(1-x).\)
These kind of expectations, sometimes also called distorted expectations (w.r.t. the distortion function $g$),  are of particular interest if $g$ is concave. In this case they were suggested in insurance for premium calculation (see, e.g. \cite{Denneberg1994}) and justified by some axiomatization of insurance pricing provided in \cite{wang1997axiomatic}. For concave $g$, distorted expectations were also used in finance to model bid-ask spreads, see 
\cite{cherny2009new} or \cite{madan2012structured}  for static versions and  
\cite{bielecki2012dynamic} for dynamic extensions. If $g$ is continuous and concave, then the distorted expectation has the following representation
\begin{equation}
\label{cErobust}
\cE^{g}(X) = \sup_{\qr\in \sigma-\text{core}(g(\pr))}\ex_{\qr}[X],
\end{equation}
where $\sigma-\operatorname{core}(g(\pr))$ consists of all probability measure $\qr$ on 
$\cF$ satifying $\qr(A)\leq g(\pr(A))$ for any $A\in\cF$ (see e.g. \cite[Proposition 10.3 with Example 2.1]{Denneberg1994}). In view of \eqref{cErobust}, the distorted expectations may be interpreted as expectations under  model uncertainty induced by the set \(\sigma-\text{core}(g(\pr)).\)
\par
In this paper we are also going to study more general types of nonlinear functionals related to law-invariant coherent risk measures.
Consider  the space \(L^p(\Omega,\cF, \pr),\) \(p\in [1,\infty),\) of measurable functions \(X: \Omega\to \R\) (random variables) having finite \(p\)th order moment; for \(p=\infty\)  the space \(L^\infty(\Omega,\cF, \pr)\) is formed by essentially bounded measurable functions. Let \(\cX\) be a vector space  such that \(L^\infty(\Omega,\cF, \pr)\subseteq\cX\subseteq L^1(\Omega,\cF, \pr),\) \(X\in \cX\) implies \(|X|\in\cX\) and for any \(X\in L^1(\Omega,\cF, \pr)\)  with \(|X|\preceq |Y|\) and \(|Y|\in\cX,\) it holds \(X\in \cX.\) Here the notation \(X\preceq Y\) means that \(X(\omega)\leq Y(\omega)\) for almost every \(\omega\) with respect to \(\pr.\)  A functional \(\cE: \cX\to \R\) is called a coherent risk measure if it fulfills the following axioms:
\begin{description} 
\item [(A1)] \textsf{Monotonicity}: If \(X,X'\in \cX\) and \(X\succeq X',\) then \(\cE(X)\geq \cE(X').\) 
\item [(A2)] \textsf{Sublinearity}:
\begin{eqnarray*}
\cE(\alpha X+\beta X')\leq \alpha\cE(X)+\beta\cE(X')
\end{eqnarray*}
for all \(X,X'\in \cX\) and all \(\alpha,\beta\geq 0.\)
\item [(A3)] \textsf{Translation equivariance}: If \(a\in \R\) and \(X\in \cX,\) then \(\cE(X+a)=\cE(X)+a.\)
\item [(A4)] \textsf{Cutoff property}: For all \(X\in \cX\) with property \(X\succeq 0\)
\begin{eqnarray*}
\lim_{k\to \infty} \cE\bigl((X-k)^+\bigr)=0.
\end{eqnarray*}
\end{description}
A risk measure \(\cE:\) \(\cX\to \R\) is called law invariant if  \(\cE\) depends only on the distribution of \(X;\) i.e., if \(X\) and \(X'\) have the same distribution then \(\cE(X)=\cE(X').\) The cutoff property is automatically fulfilled if ${\cal X}$ may be equipped with a complete norm $\|\cdot\|_{{\cal X}}$ such that ${\cal X}$ is a Banach lattice w.r.t. this norm and the partial order $\succeq$ (cf. Ruszczynski and Shapiro (\cite{RuszczynskiShapiro2006}), or Cheridito and Li (\cite{cheridito2009risk})). Outstanding examples are the standard $L^{p}-$spaces $L^{p}\OFP$ equipped with the ordinary $L^{p}-$norms $\|\cdot\|_{p}$ ($p\in [1,\infty]$). Another relevant class of examples is related with the continuous concave functions. More precisely, for any 
continuous concave distortion function $g$ the set ${\cal X}_{g}$ consisting of all random variables $X$ on $\OFP$ such that $\int_{0}^{\infty} g(1-F_{|X|}(x))~dx < \infty$ is a vector space. Tacitely identifying random variables that are identical $\pr-$a.s., it is a Banach lattice w.r.t. the complete norm
\begin{equation}
\label{distortionNorm}
\|\cdot\|_{{\cal X}_{g}}: {\cal X}_{g}\rightarrow\mathbb{R},\, X\mapsto \int_{0}^{\infty} g(1-F_{|X|}(x))~dx,
\end{equation} 
and $\succeq$ (cf. \cite{Denneberg1994}, Proposition 9.5 with Proposition 9.3). 
\par
 Let us consider some examples of law invariant coherent risk measures.
\begin{example} 
The Average Value at Risk risk measure at level $\alpha\in ]0,1]$ is defined as the following functional:
$$
AV@R_{\alpha}: \, X\mapsto -\frac{1}{\alpha}\,\int_{0}^{\alpha}
F^{\leftarrow}_{X}(\beta)\,d\beta,
$$
where \(X\) is \(~\pr-\) integrable and $F^{\leftarrow}_{X}$ denotes the left-continuous quantile function of the distribution function $F_{X}$ of $X.$ Note that $AV@R_{1}(X) = \ex[- X]$  for any $\pr-$integrable $X.$ Moreover, it is easy to check that \(\cE(X)=AV@R_{\alpha}(-X)\) is a coherent law invariant risk measure of the form
$$
\cE(X) = \cE^{g_{\alpha}}(X)
=\int_{-\infty}^{0}g_{\alpha}(F_{X}(x)) + 
\int_{0}^{\infty} [1 - g_{\alpha}(F_{X}(x))]~dx
$$
holds for any $\pr-$integrable $X,$ where the mapping $g_{\alpha}: [0,1]\rightarrow [0,1]$ is defined by 
$g_{\alpha}(u) = 1\wedge (u/\alpha).$ 
\end{example}
\begin{example}
The MINMAXVAR distortions were introduced in Cherny and Madan \cite{cherny2009new} and correspond to the continuous concave distortion function of the form:
\begin{eqnarray*}
g_{p}(u)=1-\Bigl(1-u^{1/(1+p)}\Bigr)^{1+p},\quad p\geq 0.
\end{eqnarray*}
For an integer \(p,\) we have the representation \(\cE^{g_p}(X)=\E[Y]\) with 
\[
Y\sim\min\{Z_1,\ldots,Z_{p+1}\}, \quad \max\{Z_1,\ldots,Z_{p+1}\}\sim X
\]
and this explains the name of the distortion.
\end{example}
\par
It was shown by Kusuoka (\cite{Kusuoka}) that any law invariant coherent risk measure on $L^{\infty}$ can be represented as the supremum of the mixtures of \( AV@R_\alpha\)  for different values of \(\alpha:\)
\begin{eqnarray}
\label{nonlinexp_rho}
\cE(X)=\sup_{\mu\in \mathcal{M}} \int_0^1 AV@R_\alpha(-X)\, d\mu(\alpha),
\end{eqnarray}
where \(\mathcal{M}\) is a set of probability measures on \([0,1].\) In fact the functional \(\cE\) in \eqref{nonlinexp_rho} has also a representation (see \cite{Kusuoka})
\begin{eqnarray}
\label{cEG}
\cE(X)=\sup_{g\in \cG} \cE^{g}(X)
\end{eqnarray}
for a set of concave distortion functions $\cG$. The representation \eqref{cEG} may be also verified for general law invariant coherent risk measures (cf. Kr\"atschmer and Z\"ahle \cite{KraetschmerZaehle2011}), and it will play a key role in the extension of our results to law invariant coherent risk measures.

\begin{example}
Consider the absolute semideviation risk measure defined as
\begin{eqnarray*}
\cE(X)\doteq\E[X]+c\,\E\bigl\{(X-\E(X))^{+}\bigr\}
\end{eqnarray*}
with some constant \(c\in [0,1]\) and \(X\in {\cal X} \doteq L^1(\Omega, \cF, \pr).\) By taking a two point probability measure \(\mu\)
with mass \(1-c\varkappa\) at \(\alpha=1\) and mass \(c\varkappa\) at \(\alpha=\varkappa,\) we obtain the following representation   (see Shapiro \cite{shapiro2013kusuoka})
\begin{eqnarray}
\label{semidev_kus}
\cE(X)=\sup_{\varkappa\in ]0,1[}\Bigr[ (1-c\varkappa) \,AV@R_1(-X)+c\varkappa\, AV@R_{\varkappa}(-X) \Bigl].
\end{eqnarray}
For $\varkappa\in ]0,1[$,  define a continuous concave distribution function $g_{\varkappa}$ on $[0,1]$ by 
\begin{eqnarray*}
g_{\varkappa}(\alpha)=
\begin{cases}
\alpha (c(1-\varkappa)+1), & \alpha\leq \varkappa,
\\
c\varkappa (1-\alpha)+\alpha, & \varkappa<\alpha\leq 1.
\end{cases}
\end{eqnarray*} 
It satisfies 
$$
\cE^{g_{\varkappa}} (X) = (1-c\varkappa) \,AV@R_1(-X)+c\varkappa\, AV@R_{\varkappa}(-X)
\quad\mbox{for}~X\in L^{1}\OFP 
$$
so that by \eqref{semidev_kus} 
the representation \eqref{cEG} reads as follows 
\begin{eqnarray}
\label{semidev_kus1}
\cE(X) = \sup_{g\in \cG}\cE^{g}(X)\quad\mbox{with}\quad\cG\doteq\{g_{\varkappa}|\, \varkappa\in ]0,1[\}.
\end{eqnarray}
\end{example}
\par
\begin{example}
\label{Expectiles}
For $\alpha\in ]0,1[,$ an $\alpha$-expectile of $X\in L^1\OFP$ can be defined as
\begin{eqnarray*}
    \cE_\alpha(X)\doteq
      \inf\left\{x\in\R\,\mid\,\alpha\ex\left[(X-x)^{+}\right] - (1-\alpha)\ex\left[(X-x)^{-}\right] = 0\right\}.
\end{eqnarray*}
If $X$ is square integrable, then the $\alpha$-expectile of $X$ has an alternative representation
\begin{eqnarray*}
    \cE_\alpha(X)\,=\,\argmin_{x\in\R}\, \alpha\,\ex\left[\left((X-x)^{+}\right)^{2}\right] + (1-\alpha)\,\ex\left[\left((X-x)^{-}\right)\right]
\end{eqnarray*}
(cf. \cite[Example 4]{BelliniKlarMuellerRosazza-Gianin2013}). This is the genuine definition of expectiles introduced in  Newey and Powell (\cite{NeweyPowell1987}). It has been shown in \cite{BelliniKlarMuellerRosazza-Gianin2013} that for $\alpha\in ]1/2,1[,$ the $\alpha-$expectile defines a law-invariant coherent risk measure $\cE_{\alpha}$ on ${\cal X} \doteq L^{1}\OFP$ with the representation
\begin{eqnarray}
\label{Expectile_Kusuoka}
\cE_\alpha(X) = \sup_{\gamma\in [(1-\alpha)/\alpha,1]}\Bigr[ (1-\gamma) \,AV@R_{\alpha(\gamma)}(-X)+\gamma\, AV@R_{1}(-X) \Bigl],
\end{eqnarray}
where $\alpha(\gamma)\doteq (1-\alpha)(1-\gamma)/[\gamma (2\alpha - 1))]$ (cf. 
\cite[Proposition 8, Proposition 9]{BelliniKlarMuellerRosazza-Gianin2013}). Then 
\begin{eqnarray*}
g_{\alpha,\gamma}(\beta)=
\begin{cases}
\frac{\alpha \beta \gamma}{1 - \alpha}, & \beta\leq \alpha(\gamma),
\\
1 - \gamma +  \gamma \beta, & \alpha(\gamma) < \beta\leq 1
\end{cases}
\end{eqnarray*}
defines a continuous and concave distortion function for $\gamma\in [(1-\alpha)/\alpha,1]$ satisfying
$$
\cE^{g_{\alpha,\gamma}}(X) =   (1-\gamma) \,AV@R_{\alpha(\gamma)}(-X)+\gamma\, AV@R_{1}(-X)\quad\mbox{for}~X\in L^{1}\OFP. 
$$
Hence in view of \eqref{Expectile_Kusuoka}, we obtain for $\alpha\in ]1/2,1[$
\begin{eqnarray}
\label{Expectile_Kusuoka1}
\cE_{\alpha}(X) = \sup_{g\in \cG_{\alpha}}\cE^{g}(X)\quad\mbox{with}\quad\cG_{\alpha}\doteq\{g_{\alpha,\gamma}|\, \gamma\in [(1-\alpha)/\alpha,1]\}.
\end{eqnarray}

\end{example}
\par
Let \(0<T<\infty\) and let $\OFFP$  be a filtered probability space, where 
$(\cF_{t})_{t\in [0,T]}$ is a right-continuous filtration with $\cF_{0}$ containing only the sets with probability $0$ or $1$ as well as all the null sets of $\cF$. While the distorted expectations are well established in static settings, this is much less the case in dynamic
setting related to the filtration $(\cF_{t})_{t\in [0,T]}.$ The reason is that, contrary to what is the case for the standard expectations, the collection of the ``conditional
distorted expectations'' 
\begin{equation}
\label{conditionalChoquetExpectations}
\cE^g_s(X)\doteq\int_{0}^{\infty}g(1 - F_{X|\cF_s}(x))~dx-\int_{-\infty}^{0}[1-g(1-F_{X|\cF_s}(x))]~dx
\end{equation}
corresponding to the collection of ``updated'' probability measures \(F_{X|\cF_s}\) is typically time-inconsistent. For instance, it is possible
that for two  times \(s\) and \(t\) with \(s < t,\)  we have   \(\cE_t^g(X)\geq \cE_t^g(Y),\) while nevertheless at time \(s\) the conditional distorted expectation of \(Y\) is
greater than that of \(X\). Even worse, unless $g$ being the identity map on $[0,1],$ we do not find any dynamic extension $(\cE_{t})_{t\in [0,T]}$ of $\cE^{g}$ satifying 
$\cE_{s}(X)\geq \cE_{s}(Y),$ whenever $s < t$ and  \(\cE_t(X)\geq \cE_t(Y)\)  (cf. 
\cite{kupper2009representation}). 
\par
Consider now a right-continuous nonnegative adapted stochastic process $(Y_{t})$ with bounded paths, and let $\cT$ gather all finite stopping times $\tau\leq T$ w.r.t. $(\mathcal{F}_t).$ The main object of our study is the following optimal stopping problem
\begin{equation}
\label{stoppproblem}
\sup_{\tau\in\cT}\, \cE^g(Y_{\tau}) = \sup_{\tau\in\cT}\, \int_{0}^{\infty} g(1 - F_{Y_{\tau}}(x))~dx,
\end{equation}
where $F_{Y_{\tau}}$ stands for the distribution function of $Y_{\tau}.$ 
As mentioned above, the key challenge related to the problem \eqref{stoppproblem}  is that dynamic distortions  \(\cE^g_t,\) \(t\in[0,T],\) as defined in \eqref{conditionalChoquetExpectations} do not possess  the property of dynamic time consistency:
\begin{eqnarray*}
\cE^g_s\circ \cE^g_t=\cE^g_s,\quad 0\leq s < t\leq T,
\end{eqnarray*}
except the trivial case \(g(x)\equiv x\). 
Thus, the methods based on the dynamic programming principle can not be applied to solve \eqref{stoppproblem}.

The stopping problem \eqref{stoppproblem} was recently considered by Xu and Zhou \cite{xu2013optimal} under some additional assumptions. First of all, the authors allow for all finite stopping times w.r.t. to some filtered probability space $(\Omega,\cF,(\cF_{t})_{t\geq 0},\pr),$ that is, they consider infinite horizon optimal stopping problems. Secondly, they impose a special structure on the process $(Y_{t})_{t\geq 0},$ namely  it is supposed that $Y_{t} = u(S_{t})$ for  an absolutely continuous nonnegative function  $u$ on $[0,\infty[$ and for  a one-dimensional geometric Brownian motion $(S_{t})_{t\geq 0}$.  Thirdly, the authors focus on strictly increasing absolutely continuous distortion functions $g,$ so that their analysis does not cover the case of Average Value at Risk. Summing up, 
in \cite{xu2013optimal} the optimal stopping problems of the form 
\begin{eqnarray}
\label{os_distortion}
\sup_{\tau \in \mathcal{T}^{\infty}} \, \cE^{g}(u(S_\tau))
\end{eqnarray}
are studied, where $\mathcal{T}^{\infty}$ denotes the set of all finite stopping times. A crucial step  in the authors' argumentation is the reformulation of the optimal stopping problem \eqref{os_distortion} as 
\begin{eqnarray*}
\sup_{\tau \in \mathcal{T}^{\infty}} \, \cE^{g}(u(S_\tau)) 
&=& 
\sup_{F\in {\cal D}} \int_{0}^{\infty}g(1 - F(x))u'(x)\,dx \\
&=& 
\sup_{F\in {\cal D}} \int_{0}^{1} u(F^{\leftarrow}(u)) g'(1 - u)\, du,
\end{eqnarray*}
where $u'$ and $g'$ are derivatives of $u$ and $g,$ respectively, and ${\cal D}$ denotes the set of all distribution functions $F$ with a nonnegative support such that 
$\int_{0}^{\infty}(1 - F(x))\,dx\leq S_{0}.$ The main idea of the approach in \cite{xu2013optimal} is that any such distribution function may be described as the distribution function of $S_{\tau}$ for some finite stopping time $\tau\in\cT^{\infty}$ and this makes the application of the Skorokhod embedding technique possible. Hence, the results essentially rely on the special structure of the stochastic process $(Y_{t})_{t\geq 0}$ and   seem to be not extendable to stochastic processes of the form $Y_{t} = U(X_{t}),$ where $(X_{t})_{t\geq 0}$ is a multivariate Markov process. Moreover,  it remains unclear whether the analysis of  \cite{xu2013optimal} can be carried over to the case of bounded stopping times, as  the Skorokhod embedding can not be applied to  the general sets of stopping times $\cT$ (see, e.g. \cite{ankirchner2011skorokhod}). 
\par
In this paper we continue the line of research initiated in \cite{xu2013optimal} and derive several novel representations for the value of optimal stopping  under probability distortions. Unlike  \cite{xu2013optimal}, we do not restrict our analysis to some specific type of driving processes, but consider finite horizon optimal stopping problems for general stochastic  processes. This has a consequence that our results are not as explicit as ones in  \cite{xu2013optimal}. However, our representations can be used to develop efficient numerical algorithms for approximating the value of \eqref{stoppproblem}. The analysis of this paper can be also viewed as an extension of the results of Belomestny and Kr\"atschmer \cite{BelomestnyKraetschmer2014}, where optimal stopping problems for optimized certainty equivalents
$$
\cE^{\Phi}_t(X)\doteq\sup_{\Q\in\cQ_{t}}\left(\ex_{\Q}[X|\mathcal{F}_t]- \ex\left[\Phi\left(\frac{d\Q}{d\pr}\right)\big|\cF_{t}\right]\right),
$$
were considered. Here  $\Phi: [0,\infty[\rightarrow [0,\infty]$ denotes a lower semicontinuous convex mapping,
and  $\cQ_{t}$ is the set of all probability measures $\Q,$ which are absolutely continuous w.r.t. a given measure $\pr$   and $\Q= \pr$ on \(\cF_{t}.\) Let us note that the intersection of the class of optimized certainty equivalents with the class of probability distortions is very small and essentially coincides with the Average Value at Risk.
\par
The paper is organized as follows. In Section~\ref{primal} we show a general primal representation result for optimal stopping problems under probability distortions. Next a generalisation to the case of law invariant coherent risk measures is presented.
Section~\ref{add_dual} is devoted to the additive dual representation for optimal stopping problems under coherent risk measures. A problem of pricing Bermudan maxcall options under absolute semideviation risk measure is numerically analysed in Section~\ref{num}. Finally all proofs are collected in Section~\ref{proofs}. 
\section{Main results}
\label{main_results}
Define a set ${\cal X}_{g}$ to consist of all random variables $X$ on $\OFP$ such that 
$\int_{0}^{\infty}g(1 - F_{|X|}(x))~dx < \infty.$ For the distortion function $g$ we shall assume that
\begin{equation}
\label{Annahmen Young function}
g~\mbox{is continuous and concave}.
\end{equation} 
By concavity we have $g(u)\geq u$ for $u\in [0,1]$ so that every $X\in {\cal X}_{g}$ is also $\pr-$integrable. Moreover, under \eqref{Annahmen Young function}, there exists some unique probability measure $\mu_{g}$ on the ordinary Borel $\sigma-$algebra $\cB(]0,1])$  characterized by $g'(x) = \int_x^1 1/u~\mu_{g}(du)$ for $u\in ]0,1[,$ where 
$g'$ denotes the right-sided derivative of $g|_{]0,1[}$ (cf. \cite[Lemma 4.69]{FoellmerSchied2011}). 
The space of all $\mu_{g}-$integrable random variables (modulo the $\mu_{g}-$a.s. equivalence) will be denoted by $L^{1}(\mu_{g}),$ whereas $L^{1}_{+}(\mu_{g})$ gathers all nonnegative members of $L^{1}(\mu_{g}).$ With any fixed members $Z^{o}$ of 
$L^{1}_{+}(\mu_{g})$ we associate a set $L^{1}_{+}(\mu_{g},Z^{o})$ of all $Z\in L^{1}_{+}(\mu_{g})$ such that $Z(1) = 0,$ and $\inf_{\alpha\in ]0,a]}(Z(\alpha) - Z^{o}(\alpha))\geq 0$ holds for some $a\in ]0,1[.$
Set for any $\tau\in\cT~\mbox{and}~Z\in L^{1}_{+}(\mu_{g}),$
\begin{equation}
\label{Uprocess}
U_{\tau}^{g,Z} \doteq  \int_0^1\left[\frac{(Y_{\tau} - Z(\alpha))^{+}}{\alpha} + Z(\alpha)\right]~\mu_{g}(d\alpha).
\end{equation}
In addition, we define $Y^{*} \doteq \sup_{t\in [0,T]}Y_{t},$  and
$$
Z^{*}: ]0,1]\rightarrow\R,~
Z^*(\alpha)=
\bcswitch
F_{Y^{*}}^{\leftarrow}(1-\alpha)&\alpha\in ]0,1[,\\
0&\alpha = 1,
\ecswitch
$$ 
where $F_{Y^{*}}^{\leftarrow}$ stands for  a left-continuous quantile function of the distribution function $F_{Y^{*}}$ of $Y^{*}.$  Finally, denote
$$
L^{1}[Y^{*},\mu_{g}] \doteq \left\{Z\in L^{1}_{+}(\mu_{g})\mid Z(1) = 0, 
\ex\left[\int_0^1 \frac{(Y^{*}- Z)^{+}}{\alpha}~\mu_{g}(d\alpha)\right] < \infty\right\}.
$$ 
\begin{remark}
\label{constructionmeasure}
The construction of the measure \(\mu_g\) from a given continuous concave distortion function \(g\) can be described as follows (cf. \cite[proof of Lemma 4.69]{FoellmerSchied2011}). First, we define a measure \(\nu_g\) on the Borel \(\sigma\)-algebra \(\cB(]0,1])\) via
\(\nu_g(]u,1])=g'(u)\) for \(u\in ]0,1[.\) Next set
\begin{eqnarray*}
\mu_g(A)=\int_A u\,\nu_g(du),\quad A\in \cB(]0,1]). 
\end{eqnarray*}
It follows from the above definition that for any set \(A=]0,z]\) with \(z\in ]0,1[,\)
\begin{eqnarray*}
\mu_g(A)=\int_{0}^z \nu_g (]s,z])\, ds=g(z)-zg'(z),
\end{eqnarray*}
and also 
$$
\mu_{g}(\{1\}) = \nu_{g}(\{1\}) = \lim_{z\to 1-}g'(z).
$$

\end{remark}

\subsection{Primal representation}
\label{primal} 
We shall assume $Y^{*}\in \cX_{g}.$ As a result
$$
\sup_{\tau\in\cT}\cE^{g}(Y_{\tau})\leq \int_{0}^{\infty}g(1 - F_{Y^{*}}(x))~dx < \infty.
$$ 
It follows from Lemma \ref{optimizedcertaintyequivalent} (cf. Appendix \ref{AppendixAA})
$$
\int_{0}^{\infty}g(1 - F_{Y^{*}}(x))~dx 
= 
\cE^{g}(Y^{*})
= 
\ex\left[\int \left(\frac{(Y^{*}- Z^{*})^{+}}{\alpha} + Z^{*}\right)\,d\mu_g\right] < \infty.
$$
In particular
\begin{equation}
\label{wichtige Integrierbarkeit}
\ex\left[\int \frac{(Y^{*}- Z^{*})^{+}}{\alpha}~d\mu_{g}\right] < \infty\quad\mbox{and}\quad Z^{*}\in L^{1}_{+}(\mu_{g})~\mbox{with}~Z^{*}(1) = 0.
\end{equation}
Property \eqref{wichtige Integrierbarkeit} shows that under $Y^{*}\in\cX_{g},$ the set 
$L^{1}[Y^{*},\mu_{g}]$ is not empty.  
One crucial observation for what follows is that if $Z^{o}\in L^{1}[Y^{*},\mu_{g}] ,$ then 
$L^{1}_{+}(\mu_{g},Z^{o})\subseteq L^{1}[Y^{*},\mu_{g}].$ 
\medskip

The following theorem is our main result.
\begin{theorem}
\label{new_representation}
Let $(\Omega,\cF_{t},\left.\pr\right |_{\cF_{t}})$ be atomless with countably generated $\cF_{t}$ for every $t > 0,$ and let $\cZ$ be a dense subset of $\{Z\in L^{1}_{+}(\mu_{g})\mid Z(1) = 0\}$ w.r.t. the $L^{1}-$norm.  If  (\ref{Annahmen Young function}) is fulfilled and $Y^{*}\in {\cal X}_{g},$   then it holds for any \(Z^{o}\in L^{1}[Y^{*},\mu_{g}],\) 
\begin{eqnarray}
\label{main_eq_1}
\sup_{\tau\in\cT} \cE^{g}(Y_{\tau}) 
&=&
\sup\limits_{\tau\in\cT}\inf_{Z\in L^{1}_{+}(\mu_{g})\atop Z(1) = 0} 
\ex\left[U^{g,Z}_{\tau}\right]
\\
\label{main_eq_2}
&=& 
\inf_{Z\in L^{1}_{+}(\mu_{g})\atop Z(1) = 0}\sup\limits_{\tau\in\cT} 
\ex\left[U^{g,Z}_{\tau}\right]\\
\label{main_eq_3}
&=&
\inf_{Z\in \cL^{*}_{\cZ}(\mu_{g},Z^{o})}\sup\limits_{\tau\in\cT} \ex\left[U^{g,Z}_{\tau}\right], 
\end{eqnarray}
where $\cL^{*}_{\cZ}(\mu_{g},Z^{o}) \doteq\{\eins_{]0,a]}\cdot Z^{o} + Z\mid Z\in \cZ, a\in ]0,1[\}.$ In particular, we can always take \(Z^{o}=Z^{*}\) in the above representation.
\end{theorem}
The proof of Theorem \ref{new_representation} may be found in Subsection \ref{beweis new dual representation}. 
Theorem~\ref{new_representation} deals with one fixed distortion functional \(\cE^{g}.\) Due to the representation  \eqref{cEG},
we can extend the results of Theorem~\ref{new_representation} to the case of the general law invariant coherent risk measures. 
\begin{corollary}
Let \(\cE\) be any law invariant coherent risk measure, and let \(\cZ\) be a set of \(\cB(]0,1])\)-measurable mappings such that for any \(g\in \cG\) in representation \eqref{cEG}, the set  \(\cZ\) is dense in $\{Z\in L^{1}_{+}(\mu_{g})\mid Z(1) = 0\}$ w.r.t. the $L^{1}-$norm defined by $\mu_{g}$. Moreover, let \(Y^*\in \cX\) and \(Z^{o}\in \cap_{g\in\cG} L^{1}[Y^{*},\mu_{g}],\) then 
\begin{eqnarray}
\label{main_eq_coh}
\sup_{\tau\in\cT} \cE(Y_{\tau})= \sup_{g\in\cG}\inf_{Z\in \cL^{*}_{\cZ}(Z^{o})}\sup\limits_{\tau\in\cT} \ex\left[U^{g,Z}_{\tau}\right]
\end{eqnarray}
where $\cL^{*}_{\cZ}(Z^{o}) \doteq\{\eins_{]0,a]}\cdot Z^{o} + Z\mid Z\in \cZ, a\in ]0,1[\}.$
\end{corollary}
\paragraph{Discussion}
The first equality \eqref{main_eq_1} is based on the well known representation 
$$
\cE^{g}(X) = \int_0^1 AV@R_{\alpha}(-X)~\mu_{g}(d\alpha) 
= 
\int_0^1\left\{\min_{x\in\R}\ex\left[\frac{(X+x)^{+}}{\alpha} - x\right]\,\mu_{g}(d\alpha)\right\}
$$
for $\pr-$essentially bounded random variables $X$ (see, e.g., \cite[Theorem 4.70 with Lemma 4.51]{FoellmerSchied2011}). By interchanging  integration and minimization (cf. \cite[Theorem 14.60]{rockafellar1998variational}), we can represent  \eqref{stoppproblem} as a solution of some maxmin optimization problem. The second representation \eqref{main_eq_2} is the key result of our paper and shows that we can interchange \(\sup\) with \(\inf.\) The proof of this representation relies on the notion of randomized stopping times and makes use of a novel  approximation result for measurable partitions of unity by indicator functions (see Proposition \ref{Dichtheit} in Appendix \ref{AppendixC}). 
The representation \eqref{main_eq_2} can be used to approximate the solution of \eqref{stoppproblem} via solving a sequence of the standard optimal stopping problems. Finally, the equality \eqref{main_eq_3} means that one can replace the optimization over the set \(L^{1}_{+}(\mu_{g})\) by the optimization over its dense subset. \par 
Let us point out a suitable  choice for the set $\cZ$ in Theorem 
\ref{new_representation}. To this end, let us recall the notion of Bernstein polynomials. By definition, a Bernstein polynomial of degree $n$ is a function $B_{n}$ on $[0,1]$ defined by $B_{n}(x) = \sum_{i=0}^{n} b_{i} B_{i,n}(x)$ for some $b_{0},\dots,b_{n}\in\R,$ where 
\begin{equation}
\label{BernsteinMononom}
B_{i,n}(\alpha) \doteq \left(\begin{array}{c}n\\ i\end{array}\right) \alpha^{i}(1 - \alpha)^{n - i},\quad  i\in\{0,\dots,n\}, \quad n\in\N.
\end{equation}
\begin{corollary}
\label{BernsteinPolynome}
Let $(\Omega,\cF_{t},\left.\pr\right |_{\cF_{t}})$ be atomless with countably generated $\cF_{t}$ for every $t > 0,$ and let $\cZ_{B}$ consist of all mappings $\sum_{i=0}^{n-1}b_{i} B_{i,n}|_{]0,1]}$ with $n\in\N$ and $b_{0},\dots,b_{n-1}\geq 0.$  If (\ref{Annahmen Young function}) is fulfilled and $Y^{*}\in {\cal X}_{g},$  then for any $Z^{0}\in L^{1}[Y^{*},\mu_{g}],$
$$
\sup_{\tau\in\cT} \cE^{g}(Y_{\tau}) 
=
\inf_{Z\in\cZ_{B}(Z^{o})}\sup\limits_{\tau\in\cT} 
\ex\left[\int_0^1\left[\frac{(Y_{\tau} - Z(\alpha))^{+}}{\alpha} + Z(\alpha)\right]~\mu_{g}(d\alpha)\right] < \infty,
$$
where $\cZ_{B}(Z^{o}) \doteq \{\eins_{]0,a]}\cdot Z^{0} + Z\mid Z\in\cZ_{B}, a\in ]0,1[\}.$
\end{corollary}
Let us present some special cases of Theorem \ref{new_representation}. For concave $g,$ we shall denote its right-sided derivative on $]0,1[$ by $g'.$ It is non-increasing so that it might be extended to $[0,1]$ by setting 
$g'(0) \doteq \sup_{\alpha > 0}g'(\alpha)$ and $g'(1) \doteq \inf_{\alpha < 1}g'(\alpha).$ It follows from \cite[Lemma 4.69]{FoellmerSchied2011}, that $g'(0)$ is finite iff 
$\int 1/\alpha~\mu_{g}(d\alpha) < \infty.$ In the case of finite $g'(0),$ we may choose 
$Z^{o}\equiv 0$ in Theorem \ref{new_representation} and Corollary  \ref{BernsteinPolynome} to draw the following immediate conclusion.
\begin{corollary}
\label{simplifiedrepresentation}
Let $\cZ$ be a dense subset of $\{Z\in L^{1}_{+}(\mu_{g})\mid Z(1) = 0 \}$ w.r.t.  $L^{1}-$norm. If $g'(0) < \infty,$ then under the assumptions of Theorem \ref{new_representation} we have
$$
\sup_{\tau\in\cT} \cE^{g}(Y_{\tau}) 
=
\inf_{Z\in\cZ}\sup\limits_{\tau\in\cT} 
\ex\left[\int_0^1\left[\frac{(Y_{\tau} - Z(\alpha))^{+}}{\alpha} + Z(\alpha)\right]~\mu_{g}(d\alpha)\right] < \infty.
$$
\end{corollary}
The latter corollary can be easily generalized to the case of general law invariant coherent risk measures.  First note that with any coherent law invariant risk measure \(\cE\) we can associate a function \(g_{\cE}:\) \([0,1]\to \mathbb{R}\) via
\begin{eqnarray*}
g_{\cE}(\alpha)=\cE\bigl(F^{\leftarrow}_{B(\alpha)}(U)\bigr), \quad \alpha\in [0,1],
\end{eqnarray*}
where $F^{\leftarrow}_{B(\alpha)}$ denotes the left-continuous quantile function of the distribution function $F_{B(\alpha)}$ of a Bernoulli r. v. $B(\alpha)$ with parameter \(\alpha\) and \(U\) is a random variables having uniform distribution on \([0,1].\) 
If $\lim_{\alpha\to0+}g_{\cE}(\alpha) = 0,$ then it is known that any set $\cG$ in the representation \eqref{cEG} is relatively compact w.r.t. the uniform metric consisting of continuous concave distortion functions only (see Belomestny and Kr\"atschmer \cite{BelomestnyKraetschmer2012}). This continuity condition is already fulfilled if ${\cal X}$ may be equipped with a complete $\sigma-$order continuous norm $\|\cdot\|_{{\cal X}}$ such that ${\cal X}$ is a Banach lattice w.r.t. this norm and the partial order $\succeq$. To recall, a norm $\|\cdot\|_{\cX}$ on $\cX$ is said to be $\sigma-$order continuous if  
$$
\lim_{k\to\infty}\|X_{k}\|_{\cX} = 0\quad\mbox{whenever}~X_{k}\searrow 0~\pr-\mbox{a.s.}
$$
(cf. Ruszczynski and Shapiro (\cite{RuszczynskiShapiro2006}), or Cheridito and Li (\cite{cheridito2009risk})). Of course $L^{p}-$norms on $L^{p}-$spaces are $\sigma-$order continuous for $p\in [1,\infty[$. Also for any continuous concave distortion function $g$, the complete norm $\|\cdot\|_{\cX_{g}}$ on $\cX_{g}$, as defined in \eqref{distortionNorm}, satisfies $\sigma-$order continuity due to the dominated convergence theorem.
\begin{corollary}
\label{simplifiedrepresentation_cohrisk}
Let \(\cE\) be a law invariant coherent risk measure satisfying $\lim_{\alpha\to0+}g_{\cE}(\alpha) = 0$, and let us fix any representation of the form  \eqref{cEG}. If \(\sup_{\alpha\in ]0,1]} g_{\cE}(\alpha)/\alpha<\infty\)  and for any \(g\in \cG,\) the set \(\cZ\) is dense in $\{Z\in L^{1}_{+}(\mu_{g})\mid Z(1) = 0\}$ w.r.t. the $L^{1}-$norm defined by $\mu_{g}$, then 
$$
\sup_{\tau\in\cT} \cE(Y_{\tau}) 
=
\sup_{g\in \cG}\inf_{Z\in\cZ}\sup\limits_{\tau\in\cT} 
\ex\left[\int_0^1\left[\frac{(Y_{\tau} - Z(\alpha))^{+}}{\alpha} + Z(\alpha)\right]~\mu_{g}(d\alpha)\right] < \infty.
$$
\end{corollary}

Corollary \ref{simplifiedrepresentation} may be further simplified if $\mu_{g}$ has finite support. In this case, each element of  $L^{1}_{+}(\mu_{g})$ may be identified with $[0,\infty[^{m},$ where $m$ denotes the cardinality of the support of $\mu_{g}.$ Then, as an immediate consequence of Corollary \ref{simplifiedrepresentation}, we obtain the following primal representation for the optimal stopping problem 
(\ref{stoppproblem}).
\begin{corollary}
\label{representation AV@R}
Let $g$ fulfill (\ref{Annahmen Young function}), and let $\mu_{g}$ have a finite support $supp(\mu_{g}) \doteq \{\alpha_{1},\dots,\alpha_{m}\}$ with $\alpha_{m} = \max supp(\mu_{g}).$ Then under the assumptions of Theorem \ref{new_representation}, the following statements are valid.
\begin{enumerate}
\item [(i)] 
If $\alpha_{m} < 1,$ then 
$$
\sup_{\tau\in\cT}\cE^{g}(Y_{\tau}) 
= 
\inf\limits_{x_{1},\dots,x_{m}\geq 0}\sup\limits_{\tau\in\cT}
\sum\limits_{i=1}^{m}\ex\left[\frac{1}{\alpha_{i}}\, (Y_{\tau} - x_{i})^{+} + x_{i}\right]\mu_{g}(\{\alpha_{i}\}).
$$
\item [(ii)] 
If $\alpha_{m} = 1$ with $m\geq 2,$ then 
$$
\sup_{\tau\in\cT}\cE^{g}(Y_{\tau}) 
= 
\inf\limits_{x_{1},\dots,x_{m-1}\geq 0}\sup\limits_{\tau\in\cT}\left(~
\sum\limits_{i=1}^{m-1}\ex\left[\frac{1}{\alpha_{i}}\, (Y_{\tau} - x_{i})^{+} + x_{i}\right]\mu_{g}(\{\alpha_{i}\}) + \ex[Y_{\tau}]\mu_{g}(\{1\})\right).
$$
\item [(iii)]
If $\alpha_{m} = 1$ with $m = 1,$ then 
$$
\sup_{\tau\in\cT}\cE^{g}(Y_{\tau}) 
= 
\sup\limits_{\tau\in\cT} \ex\left[Y_{\tau}\right].
$$
\end{enumerate} 
\end{corollary}
\begin{remark}
\label{constructionmeasure1}
The measure \(\mu_g\) has the finite support \(\{\alpha_{1},\dots,\alpha_{m}\}\) with 
$\alpha_{i} > \alpha_{i-1}$ for $i=2,\dots,m,$ if and only if 
\(g'\) is constant on each interval \(]\alpha_{i-1},\alpha_i[,\) \(i=1,\ldots, m +1,\) with $\alpha_{0} \doteq 0$ and $\alpha_{m+1} \doteq 1$ by definition. In this case, we may draw on Remark 
\ref{constructionmeasure} to conclude
\begin{eqnarray*}
\mu_{g}(\{\alpha_1\})&=& \alpha_1\,[g'(0+)-g'(\alpha_1)]
\\
\mu_{g}(\{\alpha_i\})&=& \alpha_i\,[g'(\alpha_{i-1})-g'(\alpha_i)], \quad i=2,\ldots, m, \quad \alpha_{m} < 1
\\
\mu_{g}(\{\alpha_m\})&=& g'(\alpha_{m-1}), \quad \alpha_{m} = 1. 
\end{eqnarray*}

\end{remark}
\par
\begin{example}[Optimal stopping under absolute semideviation]
\label{semidev}
Let us turn again to the absolute semideviation risk measure
\begin{eqnarray*}
\cE(X)=\E[X]+c\,\E\bigl[(X-\E(X))^{+}\bigr].
\end{eqnarray*}
As can be easily seen, the associated function \(g_{\cE}\) has the form 
\begin{eqnarray*}
g_{\cE}(\alpha)=\alpha+c\, \alpha\, (1-\alpha)
\end{eqnarray*}
implying $\lim_{\alpha\to0+}g_{\cE}(\alpha) = 0,$ and \(\sup_{\alpha\in ]0,1]} g_{\cE}(\alpha)/\alpha=1+c.\)  Now, we may apply 
Corollary \ref{representation AV@R} along with Remark \ref{constructionmeasure1} and representation \eqref{semidev_kus1} to obtain the following representation
\begin{eqnarray}
\sup_{\tau\in\cT} \cE(Y_{\tau}) &=& \sup_{\varkappa\in ]0,1[} \inf_{x\geq 0} \sup_{\tau\in \cT} \mathrm{E}\left[c(Y_{\tau}-x)^{+}+c\varkappa x+(1-c\varkappa)Y_{\tau}\right].
\label{as_primal}
\end{eqnarray}

\end{example}
\begin{example}[Optimal stopping under expectiles] 
Let for $\alpha\in ]1/2,1[$ consider the $\alpha-$expectile 
$$
\cE_{\alpha}: L^{1}\OFP\rightarrow\R,\, X\mapsto 
\cE_\alpha(X)\doteq
      \inf\left\{x\in\R\,\mid\,\alpha\ex\left[(X-x)^{+}\right] - (1-\alpha)\ex\left[(X-x)^{-}\right] = 0\right\}.
$$
The associated distortion function $g_{\cE_{\alpha}}$ is defined by $g_{\cE_{\alpha}}(\beta) = \alpha \beta/[\beta (2\alpha - 1) + 1 - \alpha].$ In particular, 
$\lim_{\beta\to0+}g_{\cE}(\beta) = 0$ and \(\sup_{\beta\in ]0,1]} g_{\cE}(\beta)/\beta = \alpha/(1-\alpha)\). The application of 
Corollary \ref{representation AV@R} along with Remark \ref{constructionmeasure1} to 
the representation \eqref{Expectile_Kusuoka1} yields
\begin{eqnarray*}
\sup_{\tau\in\cT} \cE_{\alpha}(Y_{\tau}) 
= 
\sup_{\gamma\in [\alpha/(1-\alpha),1]} 
\inf_{x\geq 0} \sup_{\tau\in \cT}\E\Bigl[\frac{\gamma (2\alpha - 1)}{1-\alpha}\, (Y_\tau-x)^{+}\,+\, \gamma Y_{\tau}\,+\, (1-\gamma) x\Bigr].
\end{eqnarray*}
Concerning the mapping
$$
\phi: [0,\infty[\times [\alpha/(1-\alpha),1]\rightarrow\R,~(x,\gamma)\mapsto 
\sup_{\tau\in\cT}\E\Bigl[\frac{\gamma (2\alpha - 1)}{1-\alpha}\, (Y_\tau-x)^{+}\, +\,  \gamma Y_{\tau}\,+\,(1-\gamma) x\Bigr]
$$
we may verify easily that $\phi(\cdot,\gamma)$ is convex for $\gamma\in [\alpha/(1-\alpha),1],$ and that $\phi(x,\cdot)$ is concave as well as continuous. Hence we may conclude by the Ky Fan minimax theorem for convex mappings that
$$
\sup_{\gamma\in [\alpha/(1-\alpha),1]}\inf_{x\geq 0}\phi(x,\gamma) = 
\inf_{x\geq 0}\sup_{\gamma\in [\alpha/(1-\alpha),1]}\phi(x,\gamma).
$$
Hence 
\begin{eqnarray}
\sup_{\tau\in\cT} \cE_{\alpha}(Y_{\tau}) 
&=& 
\nonumber
\sup_{\gamma\in [(1-\alpha)/\alpha,1]} 
\inf_{x\geq 0} \sup_{\tau\in \cT} \E\Bigl[\frac{\gamma (2\alpha - 1)}{1-\alpha}\, (Y_\tau-x)^{+}\,+\, \gamma Y_{\tau}\,+\,(1-\gamma) x\Bigr]\\
&=& 
\inf_{x\geq 0}\sup_{\gamma\in [(1-\alpha)/\alpha,1]} 
 \sup_{\tau\in \cT} \E\Bigl[\frac{\gamma (2\alpha - 1)}{1-\alpha}\, (Y_\tau-x)^{+}\,+\, \gamma Y_{\tau}\,+\,(1-\gamma) x\Bigr].
\label{ExpectilePrimal1} 
\end{eqnarray}
\end{example}
\medskip

We may also derive an alternative simplified representation of the stopping problem \eqref{stoppproblem} in the case of $\pr-$essentially bounded $Y^{*}.$ Then the key observation is that a $\delta\eins_{]0,1[}$ for a constant $\delta$ belongs to 
$L^{1}[Y^{*},\mu_{g}]$ if $\delta\geq Y^{*}.$ Denote the space of all real-valued uniformly continuous mappings on $]0,1]$ by $C_{u}(]0,1]).$    
\begin{theorem}
\label{boundedcashflow}
Let $(\Omega,\cF_{t},\left.\pr\right |_{\cF_{t}})$ be atomless with countably generated $\cF_{t}$ for every $t > 0,$ and let $\cZ$ be a dense subset of $\{Z\in \cC_{u}(]0,1])\mid Z\geq 0, Z(1) = 0\}$  w.r.t. the supremum norm on $\cC_{u}(]0,1]).$  If  (\ref{Annahmen Young function}) is fulfilled and $Y^{*}$ is $\mu_{g}-$essentially bounded with 
$\mu_{g}-$essential supremum $|Y^{*}|_{\infty},$ then 
\begin{eqnarray*}
\sup_{\tau\in\cT} \cE^{g}(Y_{\tau}) 
&=&
\inf_{Z\in \cZ}\sup\limits_{\tau\in\cT} \ex\left[U^{g,Z}_{\tau}\right]\\ 
&=&
\inf_{Z\in L^{1}_{+}(\mu_{g},Z_{\delta})\cap\cZ}\sup\limits_{\tau\in\cT} \ex\left[U^{g,Z}_{\tau}\right]\quad
\mbox{for all}~\delta\geq |Y^*|_{\infty},
\end{eqnarray*}
where $Z_{\delta}\doteq\delta\eins_{]0,1[}$.
\end{theorem}
The proof of Theorem \ref{boundedcashflow} is delegated to Subsection \ref{proofboundedcashflow}.
\medskip

Let $\cZ_{B}$ consist of all mappings $\sum_{i=0}^{n-1}b_{i} B_{i,n}|_{]0,1]}$ with 
$n\in\N; b_{0},\dots,b_{n-1}\geq 0,$ where $B_{i,n}$ is defined as in \eqref{BernsteinMononom}. As can be easily seen, $\cZ_{B}$ is a dense subset of $\{Z\in \cC_{u}(]0,1])\mid Z\geq 0,~Z(1) = 0\}$ w.r.t. the supremum norm on $\cC(]0,1])$ by Stone-Weierstra{\ss} theorem. For any \(\delta>0,\) denote
\begin{eqnarray*}
\cZ_{B}^{\delta}
&\doteq& 
\left\{\sum_{i=0}^{n-1}b_{i} B_{i,n}|_{]0,1]}\,\Big|\, n\in\N, b_{0}\geq\delta; b_{0},\dots,b_{n-1}\geq 0\right\}.
\end{eqnarray*}
We have $L^{1}_{+}(\mu_{g},Z_{\delta})\cap \cZ_{B}\subseteq\cZ_{B}^{\delta}\subseteq\cZ_{B}$ so that
we obtain the following immediate consequence of Theorem \ref{boundedcashflow}.
\begin{corollary}
\label{boundedcashflowBernstein}
Under assumptions of Theorem~\ref{boundedcashflow}
\begin{eqnarray*}
\sup_{\tau\in\cT} \cE^{g}(Y_{\tau}) 
&=&
\inf_{Z\in \cZ_{B}^{\delta}}\sup\limits_{\tau\in\cT} \ex\left[U^{g,Z}_{\tau}\right]\quad
\mbox{for all}~\delta\geq |Y^{*}|_{\infty}.
\end{eqnarray*}
\end{corollary}
\begin{example}[Optimal stopping under MINMAXVAR]
In the case of MINMAXVAR distortions, we get from Remark \ref{constructionmeasure}
\begin{eqnarray*}
\mu_{g_p}(]0,z])&=&g_p(z)-zg'_p(z), 
\\
&=& 1-(1-z^{1/(1+p)})^{1+p}-(1-z^{1/(1+p)})^p z^{1/(1+p)} 
\end{eqnarray*}
for  \(z\in ]0,1[\), and 
$$
\mu_{g}(\{1\}) = \nu_{g}(\{1\}) = \lim_{z\to 1-} g_{p}'(z) = 
\lim_{z\to 1-}(1-z^{1/(1+p)})^p z^{-p/(1+p)} = 0.
$$
Thus $\mu_{g}$ has a Lebesgue density $f_{g}: ]0,1]\rightarrow\R,$ defined by $f_{g}(1) = 0$ and
$$
f_{g}(z) = -z g_{p}''(z) = \frac{p}{1+p}\,\big(1- z^{1/(1+p)}\big)^{p-1} z^{-p/(1 + p)}
$$
for $z\in ]0,1[.$ 
Hence in the case of essentially bounded \(Y^*,\) we have
\begin{eqnarray*}
\sup_{\tau\in\cT} \cE^{g}(Y_{\tau}) 
&=&
\inf_{Z\in \cZ^\delta_B}\sup\limits_{\tau\in\cT} \ex\left\{\int_0^1\left[\frac{(Y_{\tau} - Z(\alpha))^{+}}{\alpha} + Z(\alpha)\right]~f_{g}(\alpha)
~d\alpha\right\}
\end{eqnarray*}
for all \(\delta\geq |Y^*|_{\infty}\).
\end{example}
\subsection{Additive dual representation}
\label{add_dual}
 In this section we generalize the celebrated additive dual representation for optimal stopping problems (see Rogers \cite{rogers2002monte}) to the case of optimal stopping under distorted expectations and law invariant coherent risk measures. The result in \cite{rogers2002monte} was formulated in terms of martingales $M$ with $M_{0} = 0$ satisfying  $\sup_{t\in [0, T]}|M_{t}|\in L^{1}.$ The set of all such adapted martingales will be denoted by $\cM_{0}.$ 
\begin{theorem}
\label{dualrepresentation}
Let 
$V_{t} \doteq \esssup_{\tau\in\cT, \tau\geq t}\ex\left[U_{\tau}\,|\,\cF_{t}\right]$ be the Snell enve\-lope of an integrable right-continuous stochastic process $(U_{t})_{t\in [0,T]}$ adapted to $\OFFP.$ 
If  $\sup_{t\in [0, T]}|U_{t}|\in L^p$  for some $p > 1,$ then 
$$
V_{0} = \sup_{\tau\in \cT}\ex[U_{\tau}] = \inf_{M\in \cM_{0}}\ex\left[\sup_{t\in [0,T]}(U_{t} - M_{t})\right],
$$
where the infimum is attained for $M = M^{*}$ with $M^{*}$ being the martingale part of the Doob-Meyer decomposition of $(V_{t})_{t\in [0,T]}.$ Furtheremore it holds 
$$
\sup_{\tau\in\cT}\ex[U_{\tau}] = \sup_{t\in [0,T]}(U_{t} - M^{*}_{t})\quad\pr-\mbox{a.s.}.
$$
\end{theorem} 
Theorem~\ref{new_representation} allows us to extend the additive dual representation to the case of the stopping problem 
\eqref{stoppproblem}.  Define for any $Z\in L^{1}[Y^{*},\mu_{g}]$ the process 
$V^{g,Z}= (V^{g,Z}_{t})_{t\in [0,T]}$ via
 \[
V^{g,Z}_{t} \doteq \esssup_{\tau\in\cT, \tau\geq t}\ex\left[\left.\int_0^1 \left(\frac{(Y_{\tau} - Z(\alpha))^{+}}{\alpha} + Z(\alpha)\right)~\mu_{g}(d\alpha)\,\right|\,\cF_{t}\right].
\] 
The following additive dual representation for the stopping problem \eqref{stoppproblem} holds. 
\begin{theorem}
\label{new_dualrepresentation}
Let $(\Omega,\cF_{t},\left.\pr\right |_{\cF_{t}})$ be atomless with countably generated $\cF_{t}$ for every $t > 0,$ and let $\cZ$ be a dense subset of $\{Z\in L^{1}_{+}(\mu_{g})\mid Z(1) = 0\}$ w.r.t. the $L^{1}-$norm.  If (\ref{Annahmen Young function}) is fulfilled,  $Y^{*}\in {\cal X}_{g}$ and  $Z^{o}\in L^{1}_{+}(\mu_{g})$ with 
$Z^{o}(1) = 0$ as well as  $\ex\big[\big(\int \frac{(Y^{*}- Z^{o})^{+}}{\alpha}~\mu_{g}(d\alpha)\big)^{p}\big] < \infty$ for some real number $p > 1,$ then  
\begin{eqnarray*}
\sup_{\tau\in\cT} \cE^{g}(Y_{\tau}) 
&=&
\inf_{Z\in L^{1}_{+}(\mu_{g}, Z^{o})}\inf\limits_{M\in\cM_{0}} 
\ex\Bigl[~\sup_{t\in [0,T]}(U^{g,Z}_{t} - M_{t})\Bigr]\\
&=&
\inf_{Z\in \cL^{*}_{\cZ}(\mu_{g},Z^{o})}\inf\limits_{M\in\cM_{0}} 
\ex\Bigl[~\sup_{t\in [0,T]}(U^{g,Z}_{t} - M_{t})\Bigr]\\
&=&
\essinf_{Z\in L^{1}_{+}(\mu_{g}, Z^{o})} 
\sup_{t\in [0,T]}(U^{g,Z}_{t} - M_{t}^{g,Z})\quad\pr-\mbox{a.s.}\\
&=&
\essinf_{Z\in \cL^{*}_{\cZ}(\mu_{g}, Z^{o})\atop Z(1) = 0} 
\sup_{t\in [0,T]}(U^{g,Z}_{t} - M_{t}^{g,Z})\quad\pr-\mbox{a.s.}.
\end{eqnarray*}
Here  $M^{g,Z}$ stands for the martingale part of the Doob-Meyer decomposition of 
$V^{g,Z}.$
\end{theorem}
The proof of Theorem \ref{new_dualrepresentation} will be found in Subsection \ref{proofdualrepresentation}. 
\begin{remark}
As a special choice for $\cZ$ in Theorem \ref{new_dualrepresentation} we may select the set of all mappings $\sum_{i=1}^{n-1}b_{i} B_{i,n}|_{]0,1]}$ with $n\in\N$ and $b_{0},\dots,b_{n-1}\geq 0,$ where $B_{i,n}$ is defined as in \eqref{BernsteinMononom}. 
\end{remark}
\medskip

If $g'(0) < \infty$, then we may choose $Z^{o} \equiv 0$ in Theorem \ref{new_dualrepresentation}. The application of Corollary \ref{simplifiedrepresentation} together with Theorem \ref{new_dualrepresentation} provides us with the following additive dual representation of the stopping problem \eqref{stoppproblem}. 
\begin{corollary}
\label{dualrepresentation_utility}
Let the assumption on $g$ and $(\cF_{t})$ of 
Theorem \ref{new_representation} be fulfilled. Furthermore, let $\cZ$ be a dense subset of $\{Z\in L^{1}_{+}(\mu_{g})\mid Z(1) = 0 \}$ w.r.t. the $L^{1}-$norm. If $g'(0) < \infty,$ and if $\sup_{t\in [0,T]}|Y_{t}|^{p}$ is $\pr-$integrable for some $p > 1,$ then the following dual representation holds 
\begin{eqnarray*}
\sup_{\tau\in\cT} \cE^{g}(Y_{\tau}) 
&=&
\inf_{Z\in\cZ}\inf\limits_{M\in\cM_{0}} 
\ex\big[\sup_{t\in [0,T]}( U^{g,Z}_{t} - M_{t})
\big]\\
&=&
\inf_{Z\in\cZ}\ex\big[\sup_{t\in [0,T]}\big(U^{g,Z}_{t} - M^{*,g,Z}_{t}\big)\big]\\
&=& 
\essinf_{Z\in\cZ}\sup_{t\in [0,T]}\big(U^{g,Z}_{t} - M^{*,g,Z}_{t}\big)
\quad \pr-\mbox{a.s.}.
\end{eqnarray*}
Here 
$M^{*,g,Z}$ stands for the martingale part of the Doob-Meyer decomposition of 
$V^{g,Z}.$
\end{corollary}
For the case of law invariant coherent risk measures of the form \eqref{cEG} we have the following
\begin{corollary}
Let \(\cE\) be a coherent law invariant risk measure satisfying $\lim_{\alpha\to 0+}g_{\cE}(\alpha) = 0$, and let us fix any representation of the form  \eqref{cEG}. If \(\sup_{\alpha\in ]0,1]} g_{\cE}(\alpha)/\alpha<\infty\)  and if for any \(g\in \cG,\) the set \(\cZ\) is dense in $\{Z\in L^{1}_{+}(\mu_{g})\mid Z(1) = 0\},$ then 
$$
\sup_{\tau\in\cT} \cE(Y_{\tau}) 
=
\sup_{g\in \cG}\inf_{Z\in\cZ}\ex\big[\sup_{t\in [0,T]}\big(U^{g,Z}_{t} - M^{*,g,Z}_{t}\big)\big]
=
\esssup_{g\in \cG}\essinf_{Z\in\cZ}\sup_{t\in [0,T]}\big(U^{g,Z}_{t} - M^{*,g,Z}_{t}\big)\quad \pr-\mbox{a.s.}.
$$
\end{corollary}
\begin{example}[Optimal stopping under absolute semideviation]
Let the assumptions on $(\cF_{t})$ of Theorem \ref{new_representation} be fulfilled, and let the random variable $\sup_{t\in [0,T]} Y_{t}^{p}$ be $\pr-$integrable for some $p > 1.$ In view of \eqref{as_primal}, the dual representation for the absolute semideviation reads as follows
\begin{eqnarray}
\nonumber
&&
\sup_{\tau\in\cT}\Bigl[\E[Y_{\tau}]+c\,\E\bigl\{[Y_{\tau}-\E(Y_{\tau})]_{+}\bigr\}\Bigr]\\
&=&
\nonumber
\sup_{\varkappa\in ]0,1[} \inf_{x\geq 0}\inf_{M\in \cM_{0}} \mathrm{E}\Big[\sup_{t\in [0,T]}\left\{c(Y_{t}-x)^{+}+c\varkappa x+(1-c\varkappa)Y_{t}-M_t\right\}\Big]\\
&=&
\nonumber
\sup_{\varkappa\in ]0,1[} \inf_{x\geq 0} \mathrm{E}\Big[\sup_{t\in [0,T]}\left\{c(Y_{t}-x)^{+}+c\varkappa x+(1-c\varkappa)Y_{t}-M^{*,\varkappa,x}_t\right\}\Big]\\
&=&
\nonumber
\esssup_{\varkappa\in ]0,1[} \essinf_{x\geq 0}\sup_{t\in [0,T]}\left\{c(Y_{t}-x)^{+}+c\varkappa x+(1-c\varkappa)Y_{t}-M^{*,\varkappa,x}_t\right\}\quad \pr-\mbox{a.s.}.
\end{eqnarray}
Here 
$M^{*,\varkappa, x}$ denotes the martingale part of the Doob-Meyer decomposition of 
the Snell-envelope of the process $(c(Y_{t}-x)^{+}+c\varkappa x+(1-c\varkappa)Y_{t})_{t\in [0,T]}.$
\end{example}
\par
\begin{example}[Optimal stopping under expectiles]
Let the assumptions on $(\cF_{t})$ of Theorem \ref{new_representation} be fulfilled, and let $\sup_{t\in [0,T]} Y_{t}\in L^{p}\OFP$ for some $p\in ]1,\infty[.$ Then for $\alpha\in ]1/2,1[$ we may draw on \eqref{ExpectilePrimal1} and \eqref{ExpectilePrimal1} to obtain the following dual representation of the optimal stopping problem under the 
$\alpha-$expectile
\begin{eqnarray}
\sup_{\tau\in\cT}\cE_{\alpha}[Y_{\tau}]=
\inf_{x\geq 0}\sup_{\gamma\in [(1-\alpha)/\alpha,1]}
\inf_{M\in \cM_{0}}\ex\Bigl[~\sup_{t\in [0,T]}\Bigl(\frac{\gamma (2\alpha - 1)}{1-\alpha}\,
(Y_{t}-x)^{+}\, +\, \gamma Y_{\tau}\,+\, (1-\gamma) x- M_{t}\Bigr)\Bigr].
\end{eqnarray}

\end{example}

\section{Numerical example}
\label{num}
In this section we illustrate how our results can be applied  to pricing Bermudan-type  options under absolute semideviation risk measure.  
Specifically, we consider the model with $d$ identically distributed assets, where each underlying has dividend yield $\delta $.
The  dynamic of assets is given by
\begin{equation}
\label{Xeq}
\frac{dX_{t}^{i}}{X_{t}^{i}}=(r-\delta )dt+\sigma dW_{t}^{i},\quad i=1,\ldots,d,
\end{equation}%
where $W_{t}^{i},\,i=1,\ldots,d$, are independent one-dimensional Brownian
motions and $r,\delta ,\sigma $ are constants. At any time $t\in
\{t_{0},\ldots,t_{J}\}$ the holder of the option may exercise it and
receive the payoff
\begin{equation*}
Y_t=G(X_{t})=e^{-rt}(\max (X_{t}^{1},...,X_{t}^{d})-K)^{+}.
\end{equation*}%
Suppose  that the seller of the Bermuda option would like to protect himself against  a downside risk, i.e. against the event
\(Y_t>\E(Y_{\tau}),\) then a risk-adjusted price of the corresponding Bermudan option  can be defined  as
\begin{eqnarray}
\label{opt_stop_ex}
V=\sup_{\tau\in \cT}\left[\E[Y_{\tau}]+c\,\E\bigl\{[Y_{\tau}-\E(Y_{\tau})]_{+}\bigr\}\right],
\end{eqnarray}
where \(\cT\) is a set \(\cF\)-measurable stopping times taking values in \(\{t_0,\ldots,t_J\}.\)
Due to Example~\ref{semidev},  one can use the standard methods based on dynamic programming principle to 
solve \eqref{opt_stop_ex}.
Indeed, for any fixed \(\varkappa\) and \(x,\) the optimal value of the stopping problem 
\begin{eqnarray*}
V= \sup_{\varkappa\in ]0,1[} \inf_{x\geq 0} \sup_{\tau\in \cT} \mathrm{E}\left[c(Y_{\tau}-x)^{+}+c\varkappa x+(1-c\varkappa)Y_{\tau}\right]
\\
=\sup_{\varkappa\in ]0,1[} \inf_{x\geq 0} \sup_{\tau\in \cT} \mathrm{E}\left[\widetilde G_{x,\varkappa}(X_\tau)\right]
\end{eqnarray*}
can be, for example, numerically approximated via the well known regression methods like Longstaff-Schwartz method (see Section 7 in \cite{glasserman2004monte}). In this way one can get a (suboptimal) stopping rule 
\[
\tau_{x,\varkappa}:=\inf\Bigl\{0\leq j\leq J: \widetilde G_{x,\varkappa}(X_{t_j})\geq \widehat C_{j,x,\varkappa}(X_{t_j})\Bigr\},
\] 
where \(\widehat C_{1,x,\varkappa},\ldots,\widehat C_{J, x,\varkappa}\) are continuation values estimates. Then
\begin{eqnarray}
\label{low_bound}
V^l_N:=\sup_{\varkappa\in ]0,1[}\inf_{x\geq 0}\left\{\frac{1}{N}\sum_{n=1}^N \widetilde G_{x,\varkappa}\Bigl(X^{(n)}_{\tau_{x,\varkappa}^{(n)}}\Bigr)\right\}
\end{eqnarray}
is a low-biased estimate for \(V\).  Note that the infimum in \eqref{low_bound} can be easily computed using a simple search algorithm. An upper-biased estimate can be constructed using the well known Andersen-Broadie dual approach (see \cite{andersen2004primal}). For any fixed \(\varkappa\in  ]0,1[ \) and \(x\geq 0,\) this approach would give us a discrete time martingale \((M_j^{x,\varkappa})_{j=0,\ldots,J}\) which in turn can be used to build an upper-biased estimate via 
\begin{eqnarray}
\label{upper_bound}
V^u_N:=\sup_{\varkappa\in ]0,1[}\inf_{x\geq 0}\left\{\frac{1}{N}\sum_{n=1}^N\left[\sup_{j=0,\ldots,J}\left(\widetilde G_{x,\varkappa}(X_{t_j}) - M^{x,\varkappa, (n)}_{j}\right)\right]\right\}.
\end{eqnarray}
Note that \eqref{upper_bound} remains upper biased even if we replace the infimum of the objective function in \eqref{upper_bound} by its value at a fixed point \(x.\)
In Table~\ref{max_call_2d} we present the bounds \(V^l_N\) and \(V^u_N\) together with their standard deviations for different values of \(c.\) As to implementation details, we used \(12\) basis functions for regression (see pp. 462-463 in   \cite{glasserman2004monte}) and \(10^4\) training paths to compute 
\(\widehat C_{1,x,\varkappa},\ldots, \widehat C_{J,x,\varkappa}.\) In the dual approach of Andersen and Broadie, \(10^3\) inner simulations were done to approximate \(M^{\varkappa,x}.\) In both cases we simulated \(N=10^4\) testing paths  to compute the final estimates. From Table~\ref{max_call_2d} one can see that the price of the Bermudan option increases with \(c\) reflecting the fact that the downside risk becomes a higher weight  as \(c\) increases.
\begin{table}
\label{max_call_2d}
\caption{Bounds (with standard deviations) for \(2\)-dimensional Bermudan max-call with parameters $K=100,\,r=0.05$, $\sigma=0.2,$ $\delta=0.1$ under semideviation risk measure with parameter \(c\)}
\begin{center}
\begin{tabular}{|c|c|c|}
\hline
\(c\)  & Lower bound \(V^l_N\) & Upper bound \(V^u_N\)\\
\hline
\hline
0 & 7.94(0.116) & 8.12 (0.208) \\
0.5 & 10.31 (0.129)  & 10.63 (0.250) \\
1 & 13.27 (0.174)  & 13.81 (0.271) \\
1.5 & 15.43 (0.193)   & 16.01 (0.302) \\
\hline
\end{tabular}
\end{center}
\end{table}

\section{Main ideas of the proofs}
\label{optimalrandomizedstoppingtimes}

In order to proof Theorem \ref{new_representation}, we shall proceed as follows. 
First, by Lemma \ref{optimizedcertaintyequivalent} (cf. Appendix \ref{AppendixAA}), we obtain 
\begin{equation}
\label{Hilfsstoppproblem}
\sup_{\tau\in\cT} \cE^{g}(Y_{\tau})
= 
\sup_{\tau\in\cT}\inf_{Z\in L^{1}_{+}(\mu_{g}) \atop Z(1) = 0} 
\ex\left[\int\left[\frac{(Y_{\tau} - Z(\alpha))^{+}}{\alpha} + Z(\alpha)\right]~\mu_{g}(d\alpha)\right],
\end{equation}
The crucial part of proof of Theorem \ref{new_representation} is to show 
\begin{eqnarray}
\label{minimaxrelationship} \notag
&&
\sup_{\tau\in\cT}\inf_{Z\in L^{1}_{+}(\mu_{g}) \atop Z(1) = 0} 
\ex\left[\int\left[\frac{(Y_{\tau} - Z(\alpha))^{+}}{\alpha} + Z(\alpha)\right]~\mu_{g}(d\alpha)\right]\\
&=& 
\inf_{Z\in L^{1}_{+}(\mu_{g}) \atop Z(1) = 0}\sup_{\tau\in\cT} 
\ex\left[\int\left[\frac{(Y_{\tau} - Z(\alpha))^{+}}{\alpha} + Z(\alpha)\right]~\mu_{g}(d\alpha)\right].
\end{eqnarray}
Using Tonelli's theorem, we obtain for any $\tau\in\cT$ and every $Z\in L^{1}_{+}(\mu_{g}),$ 
\begin{multline}
\label{DarstellungZielfunktion} 
\ex\left[\int\left[\frac{(X - Z(\alpha))^{+}}{\alpha} + Z(\alpha)\right]~\mu_{g}(d\alpha)\right]\\ 
= 
\int\ex\left[\frac{(X - Z(\alpha))^{+}}{\alpha} + Z(\alpha)\right]~\mu_{g}(d\alpha)\\
= 
\int\left(\int_{Z(\alpha)}^{\infty}\frac{(1 - F_{Y_{\tau}}(x))}{\alpha}~dx + Z(\alpha)\right)~\mu_{g}(d\alpha). 
\end{multline}
Since the set $\mathbb{F} \doteq \{F_{Y_{\tau}}\mid \tau\in\cT\}$ of distribution functions 
$F_{Y_{\tau}}$ of $Y_{\tau}$ is not, in general, a convex subset of the set of distribution functions on $\R,$ we can not apply the known minimax results. The idea is to first establish (\ref{minimaxrelationship}) for the larger class of randomized stopping times, and then to show  that the optimal value coincides with the optimal value 
\(\sup_{\tau\in\cT} \cE^{g}(Y_{\tau}).\)

Let us recall the notion of randomized stopping times. By definition 
(see e.g. \cite{edgar1982compactness}), a randomized stopping time w.r.t. $\OFFP$ is 
a mapping $\tau^{r}:\Omega\times [0,1]\rightarrow [0,\infty]$ which is nondecreasing and left-continuous in the second component such that $\tau^{r}(\cdot,u)$ is a stopping time w.r.t. $(\cF_{t})_{t\in [0,T]}$ for any $u\in [0,1].$ Notice that any randomized stopping time $\taur$ is also an ordinary stopping time w.r.t. the enlarged filtered probability space 
$\big(\Omega\times [0,1],\cF\otimes \cB([0,1]),
\big(\cF_{t}\otimes \cB([0,1])\big)_{t\in [0,T]},\pr\otimes \pr^{U}\big).$ Here 
$\pr^{U}$ denotes the uniform distribution on $[0,1],$ defined on $\cB([0,1]),$ the usual Borel $\sigma-$algebra on $[0,1].$ We shall call a randomized stopping time $\taur$ to be degenerated if $\taur(\omega,\cdot)$ is constant for every $\omega\in\Omega.$ There is an obvious one-to-one correspondence between stopping times and degenerated randomized stopping times.
\par
Consider the stochastic process $(Y_{t}^{r})_{t\geq 0},$ defined by 
$$
Y_{t}^{r} :\Omega\times [0,1]\rightarrow\R,\, (\omega,u)\mapsto Y_{t}(\omega).
$$
which is adapted w.r.t. the enlarged filtered probability space. 
Denoting by $\cTr$ the set of all randomized stopping times $\tau^{r}\leq T,$ we shall study the following new stopping problem
\begin{equation}
\label{randomstop}
\mbox{maximize } 
\cE^{g}(Y^{r}_{\taur})
~\mbox{over}~\tau^r\in\cTr.
\end{equation}
Obviously, $\cE^{g}(Y_{\tau}) = \cE^{g}(Y^{r}_{\taur})$ 
is valid for every stopping time $\tau\in\cT,$ where $\taur\in\cTr$ is the corresponding degenerated randomized stopping time such that \(\taur(\omega,u)=\tau(\omega),\) \(u\in [0,1].\) Thus, in general the optimal value of the stopping problem (\ref{randomstop}) is at least as large as the one of the original stopping problem 
(\ref{stoppproblem}) due to (\ref{Hilfsstoppproblem}). One reason to consider the new stopping problem \eqref{randomstop} is that it  has a solution under fairly general conditions.
\begin{proposition}
\label{solution}
Let \eqref{Annahmen Young function} be fulfilled, and let $\sup_{t\in [0,T]}Y_{t}\in {\cal X}_{g}.$ If $(Y_{t})_{t\in [0,T]}$ is quasi-left-continuous and if $\cF_{T}$ is countably generated, then there exists a randomized stopping time $\taur_*\in\cTr$ such that
$$
\cE^{g}(Y^{r}_{\taur_{*}})
= 
\sup_{\taur\in \cTr}\cE^{g}(Y^{r}_{\taur}).
$$
\end{proposition} 
The proof of Proposition \ref{solution} is subject of Subsection \ref{proof of solution}.
Moreover, 
the following important minimax result for the stopping problem 
(\ref{randomstop}) holds.
\begin{proposition}
\label{minimax}
If (\ref{Annahmen Young function}) is fulfilled, and if $\sup_{t\in [0,T]}Y_{t}\in {\cal X}^{g},$ then 
\begin{multline*}
\sup_{\taur\in\cTr}\inf_{Z\in L^{1}_{+}(\mu_{g}) \atop Z(1) = 0} 
\ex\left[\int\left[\frac{(Y^{r}_{\taur} - Z(\alpha))^{+}}{\alpha} + Z(\alpha)\right]~\mu_{g}(d\alpha)\right]\\
= 
\inf_{Z\in L^{1}_{+}(\mu_{g}) \atop Z(1) = 0}\sup_{\tau\in\cTr} 
\ex\left[\int\left[\frac{(Y^{r}_{\taur} - Z(\alpha))^{+}}{\alpha} + Z(\alpha)\right]~\mu_{g}(d\alpha)\right].
\end{multline*}
Moreover, if $(Y_{t})_{t\in [0,T]}$ is quasi-left-continuous and if $\cF_{T}$ is countably generated, then there exist $\tau^{r*}\in\cT^{r}$ and 
$Z^{*}\in L^{1}_{+}(\mu_{g})$ with $Z(1) = 0$ such that 
\begin{multline*}
\ex\left[\int\left[\frac{(Y^{r}_{\taur} - Z^{*}(\alpha))^{+}}{\alpha} + Z^{*}(\alpha)\right]~\mu_{g}(d\alpha)\right]\\
\leq
\ex\left[\int\left[\frac{(Y^{r}_{\tau^{r*}} - Z^{*}(\alpha))^{+}}{\alpha} + Z^{*}(\alpha)\right]~\mu_{g}(d\alpha)\right]\\
\leq
\ex\left[\int\left[\frac{(Y^{r}_{\tau^{r*}} - Z(\alpha))^{+}}{\alpha} + Z(\alpha)\right]~\mu_{g}(d\alpha)\right]
\end{multline*}
for $Z\in L^{1}_{+}(\mu_{g})$ with $Z(1) = 0$ and $\taur\in\cT^{r}.$
\end{proposition}
The proof of Proposition \ref{minimax} can be found in Subsection \ref{beweis minimax}.
In the next step we shall provide conditions  ensuring that the stopping problems 
(\ref{stoppproblem}) and (\ref{randomstop}) have the same optimal value. 
\begin{proposition}
\label{derandomize2}
Let $(\Omega,\cF_{t},\pr|_{\cF_{t}})$ be atomless with countably generated $\cF_{t}$ for every $t > 0.$ If (\ref{Annahmen Young function}) is fulfilled, and if $Y^*_{t}\in {\cal X}_{g},$ then 
\begin{eqnarray*}
\sup\limits_{\tau^{r}\in\cT^{r}}\cE^{g}(Y_{\taur}) 
=
\sup\limits_{\tau\in\cT}\cE^{g}(Y_{\tau}).  
\end{eqnarray*}
\end{proposition}
The proof of Proposition \ref{derandomize2} is delegated to Subsection \ref{beweis derandomize2}.


\section{Proofs}
\label{proofs}
We shall start with some preparations which also will turn out to be useful later on. Let us recall (cf. \cite{edgar1982compactness}) that every 
$\taur\in\cTr$ induces a 
stochastic kernel
$
K_{\taur}:\Omega\times\cB([0,T])\rightarrow [0,1]
$ 
with $K_{\taur}(\omega,\cdot)$ being the distribution of $\taur(\omega,\cdot)$ under $\pr^{U}$ for any $\omega\in\Omega.$ Here $\cB([0,T])$ stands for the usual Borel $\sigma-$algebra on $[0,T].$ This stochastic kernel has the following properties:
\begin{eqnarray*}
&& 
K_{\taur}(\cdot,[0,t])~\mbox{is}~\cF_{t}-\mbox{measurable for every}~t\geq 0,\\
&&
K_{\taur}(\omega,[0,t]) = \sup\{u\in [0,1]\mid \taur(\omega,u)\leq t\}.
\end{eqnarray*}
The  associated stochastic kernel $K_{\taur}$ is useful to characterize the distribution function $F_{Y^{r}_{\taur}}$ of $Y_{\taur}^{r}.$ 
\begin{lemma}
\label{stopped distribution}
For any $\taur\in\cTr$ with associated stochastic kernel $K_{\taur},$ the distribution function $F_{Y_{\taur}^{r}}$ of $Y_{\taur}^{r}$ may be represented in the following way
$$
F_{Y_{\taur}^{r}}(x) = \ex[K_{\taur}(\cdot,\{t\in [0,T]\mid Y_{t}\leq x\})]\quad\mbox{for}~ x\in\R.
$$
\end{lemma}
\begin{proof}
~cf. \cite[Lemma 7.1]{BelomestnyKraetschmer2014}.
\end{proof}


\subsection{Proof of Proposition \ref{minimax}}
\label{beweis minimax}
The random variable $Y^{*}$ is assumed to belong to $\cX_{g}.$ In particular $Y^{r}_{\taur}\in\cX_{g}$ for $\taur\in\cT^{r},$ and in view of Lemma \ref{optimizedcertaintyequivalent} (cf. Appendix \ref{AppendixAA}), we have  
\begin{eqnarray}
\label{Hilfsstoppproblem2} \notag
&&
\cE^{g}(Y^{r}_{\taur})\\ 
&=& 
\inf_{Z\in L^{1}_{+}(\mu_{g}) \atop Z(1) = 0} 
\ex\left[\int\left[\frac{(Y^{r}_{\taur} - Z(\alpha))^{+}}{\alpha} + Z(\alpha)\right]~\mu_{g}(d\alpha)\right]\\ \notag
&=&
\ex\left[\int\eins_{]0,1[}(\alpha)\left[\frac{(Y^{r}_{\taur} - F^{\leftarrow}_{Y^{r}_{\taur}}(1 - \alpha))^{+}}{\alpha} + F^{\leftarrow}_{Y^{r}_{\taur}}(1 - \alpha)\right]~\mu_{g}(d\alpha)\right]\\ \notag
&&
+ \mu_{g}(\{1\})\ex[X],
\end{eqnarray}
where $F^{\leftarrow}_{Y^{r}_{\taur}}$ denotes the left-continuous quantile function of the distribution function $F_{Y^{r}_{\taur}}$ of $Y^{r}_{\taur}.$ 
Also by Lemma \ref{optimizedcertaintyequivalent} we obtain
\begin{eqnarray*}
\infty 
&>& 
\ex\left[\int\eins_{]0,1[}(\alpha)\left[\frac{(Y^{*} - F^{\leftarrow}_{Y^{*}}(1 - \alpha))^{+}}{\alpha} + 
F^{\leftarrow}_{Y^{*}}(1 - \alpha)\right]~\mu_{g}(d\alpha)\right]\\ 
&\geq& 
\int\eins_{]0,1[}F^{\leftarrow}_{Y^{*}}(1 - \alpha)~\mu_{g}(d\alpha),
\end{eqnarray*}
where $F^{\leftarrow}_{Y^{*}}$ stands for the left-continuous quantile function of the distribution function $F_{Y^{*}}$ of $Y^{*}.$ In particular 
$\eins_{]0,1[}F^{\leftarrow}_{Y^{*}}(1-\cdot)\in L^{1}_{+}(\mu_{g}).$ Since 
the inequality $F^{\leftarrow}_{Y^{r}_{\taur}}(\alpha)\leq F^{\leftarrow}_{Y^{*}}(\alpha)$ is valid for every 
$\taur\in\cTr$ and any $\alpha$ from $]0,1[,$ we may conclude that
\begin{equation}
\label{stoppingcompact}
\cE^{g}(Y^{r}_{\taur}) = \inf_{Z\in \cK_{Y^{*}}} 
\ex\left[\int\left[\frac{(Y^{r}_{\taur} - Z(\alpha))^{+}}{\alpha} + Z(\alpha)\right]~\mu_{g}(d\alpha)\right]
\end{equation}
holds for $\taur\in\cTr,$ where 
$$
\cK_{Y^{*}} \doteq \{Z\in L^{1}_{+}(\mu_{g})\mid  Z = \eins_{]0,1[} Z\leq \eins_{]0,1[}F^{\leftarrow}_{Y^{*}}(1 - \cdot)~ \mu_{g}-\mbox{a.s.}\}.
$$
Let us define the mapping $h: \cTr\times \cK_{Y^{*}}\rightarrow [0,\infty]$ by
$$
h(\taur,Z) \doteq \ex\left[\int\left[\frac{(Y^{r}_{\taur} - Z(\alpha))^{+}}{\alpha} + Z(\alpha)\right]~\mu_{g}(d\alpha)\right].
$$
Since $\eins_{]0,1[}F^{\leftarrow}_{Y^{*}}(1-\cdot)\in L^{1}_{+}(\mu_{g}),$ the set $\cK_{Y^{*}}$ is uniformly $\mu_{g}-$integrable, in particular it is a relatively weakly compact subset of $L^{1}(\mu_{g})$ by Dunford-Pettis theorem. Moreover, $\cK_{Y^{*}}$ is convex and closed w.r.t. the $L^{1}-$norm so that it is also weakly closed. Thus 
\begin{equation}
\label{ersteBedingungMinimax}
\cK_{Y^{*}}~\mbox{is a weakly compact subset of}~ L^{1}(\mu_{g}).
\end{equation}
Next we may observe directly from \eqref{stoppingcompact} that
\begin{equation}
\label{zweiteBedingungMinimax}
h(\taur,\cdot)~\mbox{is a proper convex function for every}~\taur\in\cTr. 
\end{equation}
Using Tonelli's theorem, we obtain for any $\taur\in\cT^{r}$ and every $Z\in L^{1}_{+}(\mu_{g})$ 
\begin{eqnarray}
\label{DarstellungZielfunktion} \notag
&&
\ex\left[\int\left[\frac{(Y^{r}_{\taur} - Z(\alpha))^{+}}{\alpha} + Z(\alpha)\right]~\mu_{g}(d\alpha)\right]\\ \notag
&=& 
\int\ex\left[\frac{(Y^{r}_{\taur} - Z(\alpha))^{+}}{\alpha} + Z(\alpha)\right]~\mu_{g}(d\alpha)\\
&=& 
\int\left(\int_{Z(\alpha)}^{\infty}\frac{(1 - F_{Y^{r}_{\taur}}(x))}{\alpha}~dx + Z(\alpha)\right)~\mu_{g}(d\alpha) 
\end{eqnarray}
Applying the monotone convergence theorem, we may rewrite  $h$ in the following way.
\begin{equation}
\label{rewrite} 
h(\taur,Z)
= 
\lim_{\varepsilon\to\infty}\int\left(\int_{Z(\alpha)}^{\varepsilon}\frac{(1 - F_{Y^{r}_{\taur}}(x))}{\alpha}~dx + Z(\alpha)\right)~\mu_{g}(d\alpha). 
\end{equation}
Moreover, for every $\taur_{1},\taur_{2}\in\cTr$ and any $\lambda\in ]0,1[,$ there exists some $\taur\in\cTr$ such that $F_{Y^{r}_{\taur}} = \lambda F_{Y^{r}_{\taur_{1}}} + 
(1 - \lambda) F_{Y^{r}_{\taur_{2}}}$ (see \cite[subsection 7.1]{BelomestnyKraetschmer2014}). Hence we may conclude from \eqref{rewrite}
\begin{equation}
\label{dritteBedingungMinimax}
\forall \taur_{1},\taur_{2}\in\cTr~\forall \lambda\in ]0,1[~\exists \taur\in\cTr: 
h(\taur,\cdot) = \lambda h(\taur_{1},\cdot) + (1-\lambda) h(\taur_{2},\cdot).
\end{equation}
We want to apply K\"onig's minimax theorem (cf. \cite[Theorem 4.9]{konig1982some}) to the mapping $h.$ In view of \eqref{ersteBedingungMinimax}, \eqref{zweiteBedingungMinimax}, 
\eqref{dritteBedingungMinimax} it remains to show that $h(\taur,\cdot)$ is weakly lower semicontinuous for any $\taur.$ For this purpose let fix an arbitrary $\taur\in\cTr.$ Since $h(\taur,\cdot)$ is a convex mapping, it suffices to show that it is lower semicontinuous w.r.t. the $L^{1}-$norm. So let $(Z_{n})_{n\in\N}$ be any sequence in 
$\cK_{Y^{*}}$ which converges to some $Z\in \cK_{Y^{*}}.$ Then for every $\varepsilon > 0,$ we obtain
\begin{eqnarray*}
&&
\left|\int\int_{Z_{n}(\alpha)}^{\varepsilon}(1 - F_{Y_{\taur}}(x))~dx ~\mu_{g}(d\alpha) 
- \int\int_{Z(\alpha)}^{\varepsilon}(1 - F_{Y_{\taur}}(x))~dx~\mu_{g}(d\alpha)\right|\\
&\leq& 
\int\left|\int_{Z_{n}(\alpha)}^{Z(\alpha)}(1 - F_{Y_{\taur}}(x))~dx\right| ~\mu_{g}(d\alpha) 
\leq 
\int\left|Z_{n}(\alpha) - Z(\alpha)\right| ~\mu_{g}(d\alpha)\to 0\quad\mbox{for}~n\to\infty. 
\end{eqnarray*}
Hence we may conclude from \eqref{rewrite}
\begin{eqnarray*}
\liminf_{n\to\infty}h(\taur,Z_{n}) 
&\geq& 
\liminf_{n\to\infty}\int\left(\int_{Z_{n}(\alpha)}^{\varepsilon}(1 - F_{Y_{\taur}}(x))~dx  + Z_{n}(\alpha)\right)~\mu_{g}(d\alpha)\\
&\geq&
\int\left(\int_{Z(\alpha)}^{\varepsilon}(1 - F_{Y_{\taur}}(x))~dx  + Z(\alpha)\right)~\mu_{g}(d\alpha).
\end{eqnarray*}
Then the application of \eqref{rewrite} again yields 
$\liminf_{n\to\infty}h(\taur,Z_{n})\geq h(\taur,Z).$ Thus $h(\taur,\cdot)$ is shown to be lower semicontinuous w.r.t. the $L^{1}-$norm so that by K\"onig's minimax theorem along with \eqref{Hilfsstoppproblem2} and \eqref{stoppingcompact}
\begin{eqnarray*}
\sup_{\taur\in\cT^{r}}\inf_{Z\in K_{Y^{*}}} h(\taur,z) 
&=&
\inf_{Z\in K_{Y^{*}}}\sup_{\taur\in\cT^{r}} h(\taur,z)\\
&\geq&
\inf_{Z\in L^{1}_{+}(\mu_{g})}\sup_{\taur\in\cT^{r}}\ex\left[\int\left[\frac{(Y^{r}_{\taur} - Z(\alpha))^{+}}{\alpha} + Z(\alpha)\right]~\mu_{g}(d\alpha)\right]\\
&\geq& 
\sup_{\taur\in\cT^{r}}\inf_{Z\in L^{1}_{+}(\mu_{g})}\ex\left[\int\left[\frac{(Y^{r}_{\taur} - Z(\alpha))^{+}}{\alpha} + Z(\alpha)\right]~\mu_{g}(d\alpha)\right]\\
&=&
\sup_{\taur\in\cT^{r}}\inf_{Z\in K_{Y^{*}}} h(\taur,z). 
\end{eqnarray*}
This shows the first part of Proposition \ref{minimax}. The second part of Proposition \ref{minimax} follows immediately from the first one along with 
\eqref{Hilfsstoppproblem2} and Proposition \ref{solution}. The proof of Proposition \ref{minimax} is complete.
\hfill$\Box$
\subsection{Proof of Proposition \ref{derandomize2}}
\label{beweis derandomize2}
The starting idea for proving Proposition \ref{derandomize2} is to reduce the stopping problem (\ref{randomstop}) to suitably discretized random stopping times. 
The choice of the discretized randomized stopping times is suggested by the following lemma (cf. \cite[Lemma 7.2]{BelomestnyKraetschmer2014}).
\begin{lemma}
\label{discretize}
For $\taur\in\cTr$ the construction
$$
\taur[j](\omega,u) \doteq \min\{k/2^{j}\mid k\in\N, \taur(\omega,u)\leq k/2^{j}\}\wedge T
$$
defines a sequence $(\taur[j])_{j\in\N}$ in $\cTr$ satisfying the following properties.
\begin{enumerate}
\item [\rm{(i)}] $\taur[j]\searrow\taur$ pointwise, in particular it follows
$$
\lim\limits_{j\to\infty}Y^{r}_{\taur[j](\omega,u)}(\omega,u) = Y^{r}_{\taur(\omega,u)}(\omega,u)
$$
for any $\omega\in\Omega$ and every $u\in [0,1].$ 
\item [\rm{(ii)}] $\lim\limits_{j\to\infty}F_{Y^{r}_{\taur[j]}}(x) = F^{r}_{Y_{\taur}}(x)$ holds for any continuity point $x$ of $F_{Y_{\taur}}.$
\item [\rm{(iii)}] For any $x\in\R$ and every $j\in\N$ we have
$$
F_{Y^{r}_{\taur[j]}}(x) = \ex\left[\widehat{Y}_{t_{1j}}^{x} K_{\taur}(\cdot,[0,t_{1j}])\right] 
+ \sum\limits_{k=2}^{\infty}\ex\left[\widehat{Y}_{t_{kj}}^{x}\, K_{\taur}(\cdot,]t_{(k-1)j},
t_{kj}])\right],
$$
where  $t_{kj} \doteq (k/2^{j})\wedge T$ for $k\in\N,$ and $\widehat{Y}^{x}_{t} \doteq \eins_{]-\infty,x]}\circ Y_{t}$ for $t\in [0,T].$
\end{enumerate}
\end{lemma}
We shall use the discretized randomized stopping times, as defined in Lemma \ref{discretize}, to show 
that we can restrict ourselves to discrete randomized stopping times in the stopping problem (\ref{randomstop}). This will be an immediate consequence of the following 
auxiliary result.
\begin{lemma}
\label{approximationstopping}
Let $(\taur_{n})_{n\in\N}$ be any sequence such that $(Y^{r}_{\taur_{n}})_{n\in\N}$ converges in law to $Y^{r}_{\taur}$ for some $\taur\in\cTr.$ 
If (\ref{Annahmen Young function}) is fulfilled, and if $Y^{*}\doteq\sup_{t\in [0,T]}Y_{t}\in 
{\cal X}_{g},$ then
$$
\lim_{n\to\infty}\cE^{g}(Y^{r}_{\taur_{n}}) = \cE^{g}(Y^{r}_{\taur}).
$$ 
\end{lemma} 
\begin{proof}
Since $g$ is continuous we have 
$$
\lim_{n\to\infty}g(1- F_{Y^{r}_{\taur_{n}}}(X)) = 
g(1 - F_{Y^{r}_{\taur}}(x))
$$ 
for every continuity point of $F_{Y^{r}_{\taur}}$ due to assumption on the convergence of $(Y^{r}_{\taur_{n}})_{n\in\N}.$ Moreover, $g(1 - F_{Y^{r}_{\taur_{n}}}(x))\leq 
g(1 - F_{Y^{*}}(x))$ holds for every $n\in\N$ and any $x > 0.$ So, by 
$Y^{*}\in {\cal X}_{g}$ we may apply the dominated convergence theorem to conclude
\begin{eqnarray*}
\lim_{n\to\infty}\cE^{g}(Y^{r}_{\taur_{n}}) = \lim_{n\to\infty}\int_{0}^{\infty} 
g(1 - F_{Y^{r}_{\taur_{n}}}(x))~dx &=& \int_{0}^{\infty}g(1 - F_{Y^{r}_{\taur}}(x))~dx 
= 
\cE^{g}(Y^{r}_{\taur}).
\end{eqnarray*}
This completes the proof.
\end{proof}
Combining Lemma \ref{approximationstopping} with Lemma \ref{stopped distribution} we obtain the following corollary.
\begin{corollary}
\label{discretizedstop}
If (\ref{Annahmen Young function}) is fulfilled, and 
$Y^{*}\doteq\sup_{t\in [0,T]}Y_{t}\in 
{\cal X}_{g},$ then for any $\taur\in\cTr$ we have 
$$
\cE^{g}(Y^{r}_{\taur}) = \lim_{j\to\infty}\cE^{g}(Y^{r}_{\taur[j]}).
$$
\end{corollary}
The following result provides the remaining missing link to prove Proposition \ref{derandomize2}.
\begin{lemma}
\label{missinglink}
Let (\ref{Annahmen Young function}) be fulfilled. Furthermore, let $\taur\in\cTr,$ and let us for any $j\in\N$ denote by $\cT[j]$ the set containing all nonrandomized stopping times from $\cT$ taking  values in 
the set $\{(k/2^{j})\wedge T\mid k\in\N\}$ with probability \(1.\) If $(\Omega,\cF_{t},\pr|_{\cF_{t}})$ is atomless with countably generated $\cF_{t}$ for every $t > 0,$ and if $Y_{t}\in {\cal X}_{g}$ for $t > 0,$ then there exists some sequence 
$(\tau_{jn})_{n\in\N}$ in $\cT[j]$ such that $(Y_{\tau_{jn}})_{n\in\N}$ converges in law to 
$Y^{r}_{\taur[j]}.$ In particular
\begin{eqnarray}
\label{tauj_minimax}
\cE^{g}(Y^{r}_{\taur[j]}) 
\leq 
\sup_{\tau\in\cT[j]}\cE^{g}(Y_{\tau}).
\end{eqnarray}
\end{lemma}
\begin{proof}
Let $k_{j} \doteq \min\{k\in\N\mid k/2^{j}\geq T\}.$ If $k_{j} = 1,$ then the statement of 
Lemma \ref{missinglink} is obvious. So let us assume $k_{j}\geq 2,$ and set $t_{kj} \doteq (k/2^{j})\wedge T.$ We already know from Lemma \ref{discretize} that 
\begin{equation}
\label{eins}
\hspace*{-0.25cm}F_{Y^{r}_{\taur[j]}}(x) = \ex\left[\widehat{Y}_{t_{1j}}^{x} K_{\taur}(\cdot,[0,t_{1j}])\right] 
+ \sum\limits_{k=2}^{k_{j}}\ex\left[\widehat{Y}_{t_{kj}}^{x} K_{\taur}(\cdot,]t_{(k-1)j},t_{kj}])\right]
\end{equation}
holds for any $x\in\R.$ Here $\widehat{Y}^{x}_{t} \doteq \eins_{]-\infty,x]}\circ Y_{t}$ for $t\in [0,T].$ 
Next
$$
Z_{k} \doteq 
\bcswitch
K_{\taur}(\cdot,[0, t_{1j}])&k = 1\\
K_{\taur}(\cdot,]t_{(k-1) j},t_{kj}])&k\in\{2,...,k_{j}\}
\ecswitch
$$
defines a random variable on $(\Omega,\cF_{t_{kj}},\pr_{|\cF_{t_{kj}}})$ which satisfies 
$0\leq Z_{k}\leq 1$ $\pr-$a.s.. In addition, we may observe that $\sum_{k=1}^{k_{j}}Z_{k} = 1$ holds $\pr-$a.s.. Since the probability spaces 
$\OFPk$ $(k=1,\dots,k_{j})$ are assumed to be atomless and countably generated, we may draw on Proposition \ref{Dichtheit} (cf. Appendix \ref{AppendixC}) along with Lemma 
\ref{BanachAlaoglu} (cf. Appendix \ref{AppendixC}) and Proposition \ref{angelic} (cf. Appendix \ref{AppendixA}) to find a sequence 
$\big((B_{1n},\dots,B_{k_{j}n})\big)_{n\in\N}$ in $\Timesk\cF_{t_{kj}}$ such that $B_{1n},\dots,B_{k_{j}n}$ is a partition of $\Omega$ for $n\in\N,$ and
$$
\lim_{n\to\infty}\ex\left[\eins_{B_{kn}}\cdot f\right] = \ex\left[Z_{k}\cdot f\right]
$$
holds for $\pr_{|\cF_{t_{kj}}}-$integrable random variables $f$ and $k\in\{1,\dots,k_{j}\}.$ In particular we have by (\ref{eins}) 
$$
F_{Y^{r}_{\taur[j]}}(x) =
\lim\limits_{n\to\infty} \sum\limits_{k=1}^{k_{j}}\ex\left[\widehat{Y}_{t_{kj}}^{x} 
\eins_{B_{kn}}\right]\, \mbox{for}\, x\in\R.
$$
We can define a sequence $(\tau_{jn})_{n\in\N}$ of nonrandomized stopping times from $\cT[j]$ via
$$
\tau_{jn} \doteq \sum\limits_{k=1}^{k_{j}}t_{kj}\, \eins_{B_{kn}}.
$$
The distribution function $F_{Y_{\tau_{jn}}}$ of $Y_{\tau_{jn}}$ satisfies
$$
F_{Y_{\tau_{jn}}}(x) = 
\sum\limits_{k=1}^{k_{j}}\ex\left[\widehat{Y}_{t_{kj}}^{x} 
\eins_{B_{kn}}\right]\, \mbox{for}\, x\in\R
$$
so that $F_{Y^{r}_{\taur[j]}}(x) = \lim_{n\to\infty}F_{Y_{\tau_{jn}}}(x)$ for $x\in\R.$
\par

The remaining part of Lemma \ref{missinglink} follows from Lemma \ref{approximationstopping}.

\end{proof}
Now, we are ready to prove Proposition \ref{derandomize2}. Putting Lemma \ref{missinglink} and Corollary \ref{discretizedstop} 
together we have 
\begin{eqnarray*}
\sup_{\tau\in\cT}\cE^{g}(Y_{\tau})\geq \sup_{\taur\in\cT^{r}}\lim_{j\to\infty}\cE^{g}(Y^{r}_{\taur[j]}) = \sup_{\taur\in\cT^{r}}\cE^{g}(Y^{r}_{\taur}),
\end{eqnarray*}
and thus
\begin{eqnarray*}
\sup_{\taur\in\cT}\cE^{g}(Y^{r}_{\taur})
\geq 
\sup_{\tau\in\cT}\cE^{g}(Y_{\tau})
\geq 
\sup_{\taur\in\cT}\cE^{g}(Y^{r}_{\taur}),
\end{eqnarray*}
completing the proof of  Proposition \ref{derandomize2}.
\hfill$\Box$

 \subsection{Proof of Theorem \ref{new_representation}}
\label{beweis new dual representation}
Firstly, we get from  Propositions \ref{minimax} and \ref{derandomize2} along with 
\eqref{Hilfsstoppproblem2}
\begin{eqnarray*}
\inf_{Z\in L^{1}_{+}(\mu_{g}) \atop Z(1) = 0}\sup_{\taur\in\cTr} 
\ex\left[\int\left[\frac{(Y^{r}_{\taur} - Z(\alpha))^{+}}{\alpha} + Z(\alpha)\right]~\mu_{g}(d\alpha)\right]
&=&
\sup_{\taur\in\cTr}\cE^{g}(Y^{r}_{\taur})\\
&=& 
\sup_{\tau\in\cT}\cE^{g}(Y_{\tau}). 
\end{eqnarray*}
Furthermore in view of \eqref{Hilfsstoppproblem}, 
\begin{eqnarray*}
\sup_{\tau\in\cT}\cE^{g}(Y_{\tau}) 
&\leq&
\inf_{Z\in L^{1}_{+}(\mu_{g}) \atop Z(1) = 0}\sup_{\tau\in\cT} 
\ex\left[\int\left[\frac{(Y_{\tau} - Z(\alpha))^{+}}{\alpha} + Z(\alpha)\right]~\mu_{g}(d\alpha)\right]\\
&\leq&
\inf_{Z\in L^{1}_{+}(\mu_{g}) \atop Z(1) = 0}\sup_{\taur\in\cTr} 
\ex\left[\int\left[\frac{(Y^{r}_{\taur} - Z(\alpha))^{+}}{\alpha} + Z(\alpha)\right]~\mu_{g}(d\alpha)\right].
\end{eqnarray*}
Thus 
\begin{eqnarray}
\label{fastEnde}\notag
\sup_{\tau\in\cT}\cE^{g}(Y_{\tau}) 
&=&
\inf_{Z\in L^{1}_{+}(\mu_{g}) \atop Z(1) = 0}\sup_{\tau\in\cT} 
\ex\left[\int\left[\frac{(Y_{\tau} - Z(\alpha))^{+}}{\alpha} + Z(\alpha)\right]~\mu_{g}(d\alpha)\right]\\ \notag
&\leq&
\inf_{Z\in L^{1}_{+}(\mu_{g},Z^{o}) \atop Z(1) = 0}\sup_{\tau\in\cT} 
\ex\left[\int\left[\frac{(Y_{\tau} - Z(\alpha))^{+}}{\alpha} + Z(\alpha)\right]~\mu_{g}(d\alpha)\right]\\ 
&\leq& 
\inf_{Z\in \cL^{*}_{\cZ}(\mu_{g},Z^{o})}\sup_{\tau\in\cT} 
\ex\left[\int\left[\frac{(Y_{\tau} - Z(\alpha))^{+}}{\alpha} + Z(\alpha)\right]~\mu_{g}(d\alpha)\right],
\end{eqnarray}
where the inequalities follow from 
$$
\cL^{*}_{\cZ}(\mu_{g},Z^{o})\subseteq \{Z\in L^{1}_{+}(\mu_{g},Z^{o})\mid Z(1) = 0\}\subseteq\{Z\in L^{1}_{+}(\mu_{g}) \mid Z(1) = 0\}.
$$ 
\medskip

Now, let $Z\in L^{1}_{+}(\mu_{g})$ with $Z(1) = 0,$ and $\varepsilon > 0.$ Then $\hZ_{k} \doteq \eins_{]0,1/k]}\cdot Z^{0} + Z$ defines a sequence $(\hZ_{k})_{k\in\N}$ in $L^{1}_{+}(\mu_{g})$ satisfying $\hZ_{k}\geq Z$ $\mu_{g}-$a.s. for $k\in\N,$ and thus 
$$
\sup_{\tau\in\cT}\ex\left[U^{g,\hZ_{k}}_{\tau}\right]
\leq
\sup_{\tau\in\cT}\ex\left[U^{g,Z}_{\tau}\right] + \int\eins_{]0,1/k]}\cdot Z^{o}~d\mu_{g}. 
$$
Since $Z^{0}\in L^{1}_{+}(\mu_{g})$ by 
assumption, 
the application of the dominated convergence theorem yields $\int\eins_{]0,1/k]}\cdot Z^{0}~d\mu_{g}\to 0,$ and hence
\begin{equation}
\label{erste Stufe}
\sup_{\tau\in\cT}\ex\left[U^{g,\hZ_{k_{0}}}_{\tau}\right]
\leq
\sup_{\tau\in\cT}\ex\left[U^{g,Z}_{\tau}\right] + \frac{\varepsilon}{2}\quad\mbox{for some}~k_{0}\in\N.
\end{equation}
Morever, by assumption, we may find some sequence $(Z_{n})_{n\in\N}$ in $\cZ$ such that 
$\int |Z_{n} - Z|~d\mu_{g}\to 0$ as well as $Z_{n}\to Z$ $\mu_{g}-$a.s.. Then $\tZ_{k_{0},n} \doteq \eins_{]0,1/k_{0}]}\cdot Z^{o} + Z_{n}$ defines a sequence $(\tZ_{k_{0},n})_{n\in\N}$ in $\cL^{*}_{\cZ}(\mu_{g},Z^{o})$ satisfying
\begin{equation}
\label{Hilfskonvergenz}
\int |\tZ_{k_{0},n} - \hZ_{k_{0}}|~d\mu_{g}\to 0\quad\mbox{and}\quad 
\tZ_{k_{0},n}\to \hZ_{k_{0}}~\mu_{g}-\mbox{a.s.}.
\end{equation}
We have by Tonelli's theorem
\begin{eqnarray*}
&&
\left|~\sup_{\tau\in\cT}\ex\big[U^{g,\tZ_{k_{0}},n}_{\tau}\big]
-
\sup_{\tau\in\cT}\ex\big[U^{g,\hZ_{k_{0}}}_{\tau}\big]\right| \\
&\leq& 
\sup_{\tau\in\cT}\int\frac{1}{\alpha}\cdot\big|\ex\big[(Y_{\tau} - \tZ_{k_{0},n}(\alpha))^{+}\big]- \ex\big[(Y_{\tau} - \hZ_{k_{0}}(\alpha))^{+}\big]\big|~\mu_{g}(d\alpha)\\
&&+ ~\int |\tZ_{k_{0},n} - \hZ_{k_{0}}|~d\mu_{g}\\
&=&
\sup_{\tau\in\cT}\int\frac{1}{\alpha}\cdot\left|\int_{\hZ_{k_{0}}(\alpha)}^{\tZ_{k_{0},n}(\alpha)} (1 - F_{Y_{\tau}}(x))~dx\right|~\mu_{g}(d\alpha) + \int|\tZ_{k_{0},n} - \hZ_{k_{0}}|~d\mu_{g}\\
&\leq&
\int\frac{1}{\alpha}\cdot\left|\int_{\hZ_{k_{0}}(\alpha)}^{\tZ_{k_{0},n}(\alpha)} (1 - F_{Y^{*}}(x))~dx\right|~\mu_{g}(d\alpha) + \int|\tZ_{k_{0},n} - \hZ_{k_{0}}|~d\mu_{g}.
\end{eqnarray*}
Here $F_{Y_{\tau}}$ and $F_{Y^{*}}$ denote respectively the distribution functions of 
$Y_{\tau}$ and $Y^{*}.$ Moreover, $\int_{0}^{\infty}(1 - F_{Y^{*}}(x))~dx < \infty$ because $Y^{*}\in\cX_{g}.$ Then in view of \eqref{Hilfskonvergenz} we obtain 
$$
\frac{1}{\alpha}\cdot\left|\int_{\hZ_{k_{0}}(\alpha)}^{\tZ_{k_{0},n}(\alpha)} (1 - F_{Y^{*}}(x))~dx\right|\leq\frac{1}{\alpha}~|\widetilde{Z}_{k_{0},n}(\alpha) - \widehat{Z}_{k_{0}}(\alpha)|\to 0\quad\mbox{for}~\mu_{g}-\mbox{almost all}~\alpha\in ]0,1].
$$
In addition
\begin{eqnarray*}
&&
\frac{1}{\alpha}\cdot\left|\int_{\hZ_{k_{0}}(\alpha)}^{\tZ_{k_{0},n}(\alpha)} (1 - F_{Y^{*}}(x))~dx\right|\\ 
&=& 
\frac{\eins_{]0,k_{0}]}(\alpha)}{\alpha}\cdot\Big|\int_{Z^{o}(\alpha) + \hZ_{k_{0}}(\alpha)}^{Z^{o}(\alpha) + \tZ_{k_{0},n}(\alpha)} (1 - F_{Y^{*}}(x))~dx\Big| 
+ 
\frac{\eins_{]k_{0},1]}(\alpha)}{\alpha}\cdot\Big|\int_{\hZ_{k_{0}}(\alpha)}^{\tZ_{k_{0},n}(\alpha)} (1 - F_{Y^{*}}(x))~dx\Big| \\[0.1cm]
&\leq& 
\frac{\eins_{]0,k_{0}]}(\alpha)}{\alpha}\cdot\Big|\int_{Z^{o}(\alpha)}^{\infty} (1 - F_{Y^{*}}(x))~dx\Big| 
+ 
\frac{\eins_{]k_{0},1]}(\alpha)}{k_{0}}\cdot\Big|\int_{0}^{\infty} (1 - F_{Y^{*}}(x))~dx\Big|\\[0.1cm]
&=&
\eins_{]0,k_{0}]}(\alpha)\cdot\frac{\ex\left[(Y^{*} - Z^{o}(\alpha))^{+}\right]}{\alpha} 
+ 
\eins_{]k_{0},1]}(\alpha)\cdot\frac{\ex\left[Y^{*}\right]}{k_{0}}
\end{eqnarray*}
Drawing on assumptions on $Z^{o},$ the application of Tonelli's theorem yields
$$
\int \frac{\ex\left[(Y^{*} - Z^{o}(\alpha))^{+}\right]}{\alpha}~d\mu_{g} 
=
\ex\left[\int\frac{(Y^{*} - Z^{o}(\alpha))^{+}}{\alpha}~\mu_{g}(d\alpha)\right] < \infty
$$
so that 
$$
\int \left(\eins_{]0,k_{0}[}(\alpha)\cdot\frac{\ex\left[(Y^{*} - Z^{o}(\alpha))^{+}\right]}{\alpha} 
+ 
\eins_{[k_{0},1]}(\alpha)\cdot\frac{\ex\left[Y^{*}\right]}{k_{0}}
\right)~\mu_{g}(d\alpha) < \infty.
$$
Hence we may apply the dominated convergence theorem to conclude 
$$
\lim_{n\to\infty}\int\frac{1}{\alpha}\cdot\left|\int_{\hZ_{k_{0}}(\alpha)}^{\tZ_{k_{0},n}(\alpha)} (1 - F_{Y^{*}}(x))~dx\right|~ d\mu_{g} = 0,
$$
and thus in view of \eqref{Hilfskonvergenz}
$$
\lim_{n\to\infty}\left|~\sup_{\tau\in\cT}\ex\big[U^{g,\tZ_{k_{0},n}}_{\tau}\big] 
- \sup_{\tau\in\cT}\ex\big[U^{g,\hZ_{k_{0}}}_{\tau}\big]\right| = 0.
$$
In particular 
$$
\sup_{\tau\in\cT}\ex\big[U^{g,\tZ_{k_{0},n_{0}}}_{\tau}\big] 
\leq \sup_{\tau\in\cT}\ex\big[U^{g,\hZ_{k_{0}}}_{\tau}\big] + \frac{\varepsilon}{2}\quad\mbox{for some}~n_{0}\in\N,
$$
and then by \eqref{erste Stufe}
$$
\sup_{\tau\in\cT}\ex\big[U^{g,\tZ_{k_{0}, n_{0}}}_{\tau}\big] 
\leq \sup_{\tau\in\cT}\ex\big[U^{g,Z}_{\tau}\big] + \varepsilon.
$$
Therefore we have shown
$$
\inf_{\tZ\in \cL^{*}_{\cZ}(\mu_{g},Z^{o})}\sup_{\tau\in\cT}\ex\big[U^{g,\tZ}_{\tau}\big] 
\leq \sup_{\tau\in\cT}\ex\big[U^{g,Z}_{\tau}\big],
$$
meaning 
\begin{eqnarray*}
\inf_{Z\in L^{1}_{+}(\mu_{g},Z^{o})\atop Z(1) = 0}\sup_{\tau\in\cT}\ex\big[U^{g,Z}_{\tau}\big] 
&\stackrel{\eqref{fastEnde}}{\leq}&
\inf_{\tZ\in \cL^{*}_{\cZ}(\mu_{g},Z^{o})}\sup_{\tau\in\cT}\ex\big[U^{g,\tZ}_{\tau}\big]
\\ 
&\leq& 
\inf_{Z\in L^{1}_{+}(\mu_{g})\atop Z(1) = 0}\sup_{\tau\in\cT}\ex\big[U^{g,Z}_{\tau}\big]
\end{eqnarray*}
because $Z$ was chosen arbitrarily. This completes the proof of Theorem \ref{new_representation} due to \eqref{fastEnde}. 
\hfill$\Box$
\subsection{Proof of Theorem \ref{boundedcashflow}}
\label{proofboundedcashflow}
Let $\widehat{\cZ} \doteq \{Z\in C_{u}(]0,1])\mid Z\geq 0, Z(1) = 0\}.$
In view of Lemma \ref{ApproximationStetigkeit} (cf. Appendix \ref{AppendixD}), $\widehat{\cZ}$ is a dense subset of $\{Z\in L^{1}_{+}(\mu_{g})\mid Z(1) = 0\}$ w.r.t. the $L^{1}-$norm generated by $\mu_{g}$. Since $\cZ$ is a dense subset of $\widehat{\cZ}$ w.r.t. the supremum norm it is also a dense subset of $\{Z\in L^{1}_{+}(\mu_{g})\mid Z(1) = 0\}$ w.r.t. the $L^{1}-$norm generated by $\mu_{g}$. If $Y^{*}$ is $\mu_{g}-$essentially bounded, then $Y^{*}\in\cX_{g}$ and $Z_{\delta} \doteq \delta\eins_{]0,1[}\in L^{1}[Y^{*},\mu_{g}]$ for 
$\delta\geq |Y^{*}|_{\infty}.$ Then by Theorem \ref{new_representation}
\begin{equation}
\label{endup}
\sup_{\tau\in\cT} \cE^{g}(Y_{\tau}) 
= 
\inf_{Z\in L^{1}_{+}(\mu_{g},Z^{o}_{\delta})\atop Z(1) = 0}\sup\limits_{\tau\in\cT} 
\ex\left[U^{g,Z}_{\tau}\right] 
= 
\inf_{Z\in L^{*}_{\widehat{\cZ}}(\mu_{g})}\sup\limits_{\tau\in\cT} 
\ex\left[U^{g,Z}_{\tau}\right],
\end{equation}
where 
$L^{*}_{\widehat{\cZ}}(\mu_{g}) \doteq \{\eins_{]0,a]}\cdot \delta + Z\mid Z\in \widehat{\cZ}, a\in ]0,1[\}.$
\medskip

Now let us fix any $\delta\geq |Y^{*}|_{\infty},$ and let $Z\in\widehat{\cZ}$ as well as 
$a\in ]0,1[.$ Then 
$$
f_{n} \doteq 
\bcswitch
\delta&0 < \alpha < a\\
n(a - \alpha) + 1&a\leq \alpha < a + \frac{1}{n}\\
0&a + \frac{1}{n}\leq\alpha \leq 1
\ecswitch
$$
defines an antitone sequence in $\cC_{u}(]0,1])$ with $f_{n}\searrow \eins_{]0,a]}.$ Then 
$(\delta f_{n} + Z)_{n\in\N}$ is an antitone sequence in $\widehat{\cZ}$ satisfying 
$\tZ_{n} \doteq (\delta + 1) f_{n} + Z\searrow \tZ,$ where $\tZ \doteq \eins_{]0,a]}\cdot Z^{o}_{\delta} + Z.$ We may observe
$$
\sup_{\tau\in\cT}\ex\left[U^{g,\tZ_{n}}_{\tau}\right] 
= 
\sup_{\tau\in\cT}\ex\left[\int \frac{\eins_{]a,1]}(\alpha)(Y_{\tau} - \tZ_{n}(\alpha))^{+}}{\alpha}~\mu_{g}(d\alpha)\right] + \int \tZ_{n}~d\mu_{g},
$$
and 
$$
\sup_{\tau\in\cT}\ex\left[U^{g,\tZ}_{\tau}\right] 
= 
\sup_{\tau\in\cT}\ex\left[\int \frac{\eins_{]a,1]}(\alpha)(Y_{\tau} - \tZ(\alpha))^{+}}{\alpha}~\mu_{g}(d\alpha)\right] + \int \tZ~d\mu_{g}.
$$
Hence 
\begin{eqnarray}
\label{direktUngleichung} \notag
&&
\left|~\sup_{\tau\in\cT}\ex\left[U^{g,\tZ_{n}}_{\tau}\right] - 
\sup_{\tau\in\cT}\ex\left[U^{g,\tZ}_{\tau}\right]\right|\\ \notag
&\leq&
\sup_{\tau\in\cT}\ex\left[\int \frac{\eins_{]a,1]}(\alpha) | (Y_{\tau} - \tZ_{n}(\alpha))^{+} - (Y_{\tau} - \tZ(\alpha))^{+}|}{\alpha}~\mu_{g}(d\alpha)\right]\\ \notag 
&&
+ 
\int |\tZ_{n} - \tZ|~d\mu_{g}\\ \notag
&\leq& 
\int \frac{\eins_{]a,1]}(\alpha) | \tZ_{n}(\alpha) - \tZ(\alpha)|}{\alpha}~\mu_{g}(d\alpha) + 
\int |\tZ_{n} - \tZ|~d\mu_{g}\\
&\leq&
\frac{1 + a}{a}\int |\tZ_{n} - \tZ|~d\mu_{g}.
\end{eqnarray}
We have $\int |\tZ_{n} - \tZ|~d\mu\to 0$ due to dominated convergence theorem so that
\begin{eqnarray*}
\left|~\sup_{\tau\in\cT}\ex\left[U^{g,\tZ_{n}}_{\tau}\right]
-
\sup_{\tau\in\cT}\ex\left[U^{g, \tZ}_{\tau}\right]\right|\to 0.
\end{eqnarray*}
Hence for arbitrary $\varepsilon$ there is some $n_{0}\in\N$ such that
\begin{equation}
\label{Vorbereitung1}
\sup_{\tau\in\cT}\ex\left[U^{g,\tZ_{n_{0}}}_{\tau}\right] 
\leq 
\sup_{\tau\in\cT}\ex\left[U^{g,\tZ}_{\tau}\right] + \frac{\varepsilon}{2}.
\end{equation}
Next, by assumption on $\cZ$ there is some sequence $(Z_{k})_{k\in\N}$ which converges to $\tZ_{n_{0}}$ w.r.t. the supremum norm. In particular 
$$
0 
= 
\lim_{k\to\infty}\sup_{\alpha\in ]0,a]}|Z_{k}(\alpha) - \tZ_{n_{0}}(\alpha)| 
= 
\lim_{k\to\infty}\sup_{\alpha\in ]0,a]}|Z_{k}(\alpha) - (\delta + 1)|. 
$$
This means that $\inf_{\alpha\in ]0,a]}Z_{k}(\alpha)\geq\delta$ for large $k.$ So we may assume without loss of generality that $\inf_{\alpha\in ]0,a]}Z_{k}(\alpha)\geq\delta$ holds for every $k\in\N,$ implying that$(Z_{k})_{k\in\N}$ is a sequence in 
$L^{1}_{+}(\mu_{g},Z_{\delta})\cap\cZ\subseteq\cZ^{\delta}.$ 
Then in the same way as in \eqref{direktUngleichung} we obtain
$$
\left|~\sup_{\tau\in\cT}\ex\left[U^{g,Z_{k}}_{\tau}\right] - 
\sup_{\tau\in\cT}\ex\left[U^{g,\tZ_{n_{0}}}_{\tau}\right]\right|\\ \notag
\leq
\frac{1 + a}{a}\int |Z_{k} - \tZ_{n_{0}}|~d\mu_{g}.
$$
Since $\int |Z_{k} - \tZ_{n_{0}}|~d\mu_{g}\leq 
\sup_{\alpha\in ]0,1]}|Z_{k}(\alpha) - \tZ_{n_{0}}(\alpha)|\to 0,$ we may find some 
$k_{0}\in\N$ such that
$$
\sup_{\tau\in\cT}\ex\left[U^{g,Z_{k_{0}}}_{\tau}\right] 
\leq 
\sup_{\tau\in\cT}\ex\left[U^{g,\tZ_{n_{0}}}_{\tau}\right] + \frac{\varepsilon}{2}.
$$
So we may conclude 
$$
\sup_{\tau\in\cT}\ex\left[U^{g,Z_{k_{0}}}_{\tau}\right] 
\leq 
\sup_{\tau\in\cT}\ex\left[U^{g,\tZ}_{\tau}\right] + \varepsilon,
$$
and thus 
$$
\sup_{\tau\in\cT}\ex\left[U^{g,\tZ}_{\tau}\right]
\geq 
\inf_{\tZ\in\cZ^{\delta}}\sup_{\tau\in\cT}\ex\left[U^{g,\tZ}_{\tau}\right]. 
$$
As $\delta\geq |Y^{*}|_{\infty}$, $a\in ]0,1[$ and $Z\in\cZ$ were chosen arbitrarily, and since 
$L^{1}_{+}(\mu_{g},Z_{\delta})\cap\cZ$ is a subset of $\{Z\in L^{1}_{+}(\mu_{g})\mid Z(1) = 0\},$ we may derive from \eqref{endup} immediately,
$$
\sup_{\tau\in\cT}\cE^{g}(Y_{\tau}) = \inf_{\tZ\in\cZ^{\delta}}\sup_{\tau\in\cT}\ex\left[U^{g,\tZ}_{\tau}\right]. 
$$
This completes the proof. 
\hfill$\Box$

\subsection{Proof of Theorem \ref{new_dualrepresentation}}
\label{proofdualrepresentation}
First of all, notice that for $Z\in L^{1}[Y^{*},\mu_{g}]$ the process $V^{g,Z}$ is nothing else but the Snell-envelope w.r.t. to the stochastic process $(U_{t}^{g,Z})_{t\in [0,T]}$ defined by \eqref{Uprocess}. 

If $Z\in L^{1}_{+}(\mu_{g})$ has the property that 
$\int\frac{(Y^{*}- Z(\alpha))^{+}}{\alpha}~\mu_{g}(d\alpha)$ is integrable of order $p$ for some $p\in [1,\infty[,$ then $\sup_{t\in [0,T]}|U_{t}^{g,Z}|$ is also integrable of order $p$ due to 
$$
\sup_{t\in [0,T]}\left|U_{t}^{g,Z}\right| 
\leq 
\int\frac{(Y^{*}- Z(\alpha))^{+}}{\alpha}~\mu_{g}(d\alpha) + \int Z~d\mu_{g}.
$$
We obtain in addition for any $\tZ\in L^{1}_{+}(\mu_{g},Z)$
\begin{eqnarray*}
&&
\ex\left[\left(\int\frac{(Y^{*}- \tZ(\alpha))^{+}}{\alpha}~\mu_{g}(d\alpha)\right)^{p}~\right]\\ 
&\leq& 
\ex\left[\left(\int\eins_{]0,a]}\frac{(Y^{*}- Z(\alpha))^{+}}{\alpha}~\mu_{g}(d\alpha)\right)^{p}~\right]\\ 
&&
+ 
\ex\left[\left(\int\eins_{]a,1]}\frac{(Y^{*}- \tZ(\alpha))^{+}}{\alpha}~\mu_{g}(d\alpha)\right)^{p}~\right]\\
&\leq&
\ex\left[\left(\int\frac{(Y^{*}- Z(\alpha))^{+}}{\alpha}~\mu_{g}(d\alpha)\right)^{p}~\right] 
+ 
\frac{1}{a}\ex\left[Y^{*}\right] < \infty,
\end{eqnarray*}
where $a\in ]0,1[$ such that $\inf_{\alpha\in ]0,a]}(\tZ(\alpha) - Z(\alpha))\geq 0.$ In particular, using Tonelli's theorem,
\begin{eqnarray*}
\int \ex\left[\frac{(Y^{*}- \tZ(\alpha))^{+}}{\alpha}\right]~\mu_{g}(d\alpha)
&=&
\ex\left[\int\frac{(Y^{*}- \tZ(\alpha))^{+}}{\alpha}~\mu_{g}(d\alpha)\right]\\
&\leq&
\left(\ex\left[\left(\int\frac{(Y^{*}- \tZ(\alpha))^{+}}{\alpha}~\mu_{g}(d\alpha)\right)^{p}~\right]\right)^{\frac{1}{p}}\\ 
&<& 
\infty.
\end{eqnarray*}
Now the statement of Theorem \ref{new_dualrepresentation} follows immediately from Theorem \ref{new_representation} with Theorem \ref{dualrepresentation}.
\hfill$\Box$
\subsection{Proof of Proposition \ref{solution}}
\label{proof of solution}
Let us introduce the filtered probability space $\tOFFP$ defined by 
$$
\widetilde{\cF}_{t} 
=
\bcswitch
\cF_{t}&t\leq T\\
\cF_{T}&t > t.
\ecswitch 
$$
We shall denote by $\widetilde{\cT}^{r}$ the set of randomized stopping times according to $\tOFFP.$ Furthermore, we may extend the processes 
$(Y_{t})_{t\in [0,T]}$ and $(Y^{r}_{t})_{t\in [0,T]}$ to right-continuous processes 
$(\widetilde{Y}_{t})_{t\in [0,\infty]}$ and $(\widetilde{Y}^{r}_{t})_{t\in [0,T]}$ in the following way
$$
\widetilde{Y}_{t} 
=
\bcswitch
Y_{t}&t\leq T\\
Y_{T}&t > t.
\ecswitch 
\quad
\mbox{and}
\quad
\widetilde{Y}^{r}_{t} 
=
\bcswitch
Y^{r}_{t}&t\leq T\\
Y^{r}_{T}&t > t.
\ecswitch. 
$$
Recall that we may equip $\widetilde{\cT}^{r}$ with the so called Baxter-Chacon topology which is compact in general, and even metrizable within our setting because $\cF_{T}$ is assumed to be countably generated (cf. Theorem 1.5 in \cite{baxter1977compactness} and discussion afterwards). 
\medskip

In the following we shall denote for any $\tilde{\tau}^{r}\in \widetilde{\cT}^{r}$ the distribution function of $\widetilde{Y}^{r}_{\tilde{\tau}^{r}}$ by 
$F_{\widetilde{Y}^{r}_{\tilde{\tau}^{r}}},$ whereas $F_{Y^{*}}$ stands for the distribution function of $Y^{*}.$ Since $F_{\widetilde{Y}^{r}_{\tilde{\tau}^{r}}}(x)\leq F_{Y^{*}}(x)$ holds for every $\tilde{\tau}^{r}\in\widetilde{\cT}^{r}$ and arbitrary $x>0,$ we obtain from $Y^{*}\in\cX_{g},$
$$
\sup_{\widetilde{\tau}^{r}}\cE^{g}(\widetilde{Y}^{r}_{\tilde{\tau}^{r}})\leq 
\int_{0}^{\infty}g(1 - F_{Y^{*}}(x))~dx < \infty.
$$
\par 

Next, consider any sequence $(\tilde{\tau}^{r}_{n})_{n\in\N}$ in $\widetilde{\cT}^{r}$ such that $\big(\cE^{g}(\widetilde{Y}^{r}_{\tilde{\tau}^{r}_{n}})\big)_{n\in\N}$ converges to 
$\sup_{\widetilde{\tau}^{r}}\cE^{g}(\widetilde{Y}^{r}_{\tilde{\tau}^{r}}).$ We may select some subsequence $(\tilde{\tau}^{r}_{i(n)})_{n\in\N}$ which converges to some 
$\tilde{\tau}^{r}\in\widetilde{\cT}^{r}$ w.r.t. the Baxter-Chacon topology due to sequential compactness of $\widetilde{\cT}^{r}$ w.r.t. this topology. By assumption on $(Y_{t})_{t\in [0,T]}$ the processes $(\widetilde{Y}_{t})_{t\in [0,\infty]}$ and $(\widetilde{Y}^{r}_{t})_{t\in [0,T]}$ are quasi-left-continuous. Hence in view of \cite[Theorem 4.7]{edgar1982compactness}, the sequence $\big(F_{\widetilde{Y}^{r}_{\tilde{\tau}^{r}_{n}}}\big)_{n\in\N}$ of distribution functions associated with the sequence 
$\big(\widetilde{Y}^{r}_{\tilde{\tau}^{r}_{n}}\big)_{n\in\N}$ satisfies
$$
F_{\widetilde{Y}^{r}_{\tilde{\tau}^{r}_{n}}}(x)\to F_{\widetilde{Y}^{r}_{\tilde{\tau}^{r}}}(x)\quad\mbox{for all}~x > 0.
$$
Due to continuity of $g$ this means
$$
g\big(1 - F_{\widetilde{Y}^{r}_{\tilde{\tau}^{r}_{n}}}(x)\big)\to 
g\big(1 - F_{\widetilde{Y}^{r}_{\tilde{\tau}^{r}}}(x)\big)\quad\mbox{for all}~x > 0.
$$
Moreover, we have
$$
\sup_{n\in\N}g\big(1 - F_{\widetilde{Y}^{r}_{\tilde{\tau}^{r}_{n}}}(x)\big)\leq 
g\big(1 - F_{Y^{*}}(x)\big)\quad\mbox{for all}~x > 0.
$$
Since $Y^{*}\in\cX_{g},$ we may apply the dominated convergence theorem to conclude
\begin{eqnarray*}
\cE^{g}(\widetilde{Y}^{r}_{\tilde{\tau}^{r}}) 
=
\int_{0}^{\infty} g\big(1 - F_{\widetilde{Y}^{r}_{\tilde{\tau}^{r}}}(x)\big)~dx
&=&
\lim_{n\to\infty}\int_{0}^{\infty}g\big(1 - F_{\widetilde{Y}^{r}_{\tilde{\tau}^{r}_{n}}}(x)\big)~dx\\
&=&
\lim_{n\to\infty}\cE^{g}(\widetilde{Y}^{r}_{\tilde{\tau}^{r}_{n}}) 
= 
\sup_{\tilde{\tau}^{r}}\cE^{g}(\widetilde{Y}^{r}_{\tilde{\tau}^{r}}).
\end{eqnarray*}
This completes the proof because 
$\tilde{Y}^{r}_{\tilde{\tau}^{r}} = Y^{r}_{\tilde{\tau}^{r}\wedge T}$ and 
$\tilde{\tau}^{r}\wedge T$ belongs to $\cTr$ for every $\tilde{\tau}^{r}\in\widetilde{\cT}^{r}.$ 
\hfill$\Box$


\section{Appendix}
\label{AppendixAA}
\begin{lemma}
\label{optimizedcertaintyequivalent}
Let the distortion function $g$ be continuous and concave. Then 
\begin{eqnarray*}
\cE^{g}(X)
&=&  
\inf_{Z\in L^{1}_{+}(\mu_{g}) \atop Z(1) = 0} 
\ex\left[\int\left[\frac{(X - Z(\alpha))^{+}}{\alpha} + Z(\alpha)\right]~\mu_{g}(d\alpha)\right]\\ 
&=& 
\ex\left[\int\eins_{]0,1[}(\alpha)\left[\frac{(X - F^{\leftarrow}_{X}(1 - \alpha))^{+}}{\alpha} + F^{\leftarrow}_{X}(1 - \alpha)\right]~\mu_{g}(d\alpha)\right]\\ 
&&
+ \mu_{g}(\{1\})\ex[X] 
\end{eqnarray*}
holds for any nonnegative $X\in {\cal X}_{g},$ where $F^{\leftarrow}_{X}$ denotes the left-continuous quantile function of the distribution function $F_{X}$ of $X.$ 

\end{lemma}
\begin{proof}
Let $X\in {\cal X}_{g}$ be nonnegative. By continuity of $g,$ monotone convergence yields
\begin{eqnarray*}
\cE^{g}(X) 
= 
\int_{0}^{\infty}g(1 - F_{X}(x))~dx 
&=& 
\lim_{k\to\infty}\int_{0}^{\infty}g(1 - F_{X\wedge k}(x))~dx\\ 
&=& 
\lim_{k\to\infty}
\cE^{g}(X\wedge k).
\end{eqnarray*}
Then by \cite[Theorem 4.70]{FoellmerSchied2011}, 
$$
\cE^{g}(X\wedge k) = \int AV@R_{\alpha}(-X\wedge k)~\mu_{g}(d\alpha)
$$
holds for every $k\in\N.$ Moreover, $AV@R_{\alpha}$ is known to be continuous from above (e.g. \cite[Theorem 4.1]{kaina2009convex}) for any $\alpha\in ]0,1]$ which means 
$$
\lim_{k\to\infty}AV@R_{\alpha}(-X\wedge k) = AV@R_{\alpha}(-X).
$$
Since $(AV@R_{\alpha}(-X\wedge k))_{k\in\N}$ is a nondecreasing sequence of nonnegative real numbers for every $\alpha \in ]0,1],$ we may conclude from the monotone convergence theorem
\begin{eqnarray}
\label{ersteDarstellung} 
\notag
\cE^{g}(X) 
&=& 
\lim_{k\to\infty} \int AV@R_{\alpha}(-X\wedge k)~\mu_{g}(d\alpha)\\ \notag
&=& 
\int AV@R_{\alpha}(-X)~\mu_{g}(d\alpha)\\ 
&=& 
\int \eins_{]0,1[}(\alpha) AV@R_{\alpha}(-X)~\mu_{g}(d\alpha) + \mu_{g}(\{1\})\ex[X]. 
\end{eqnarray}
For $\alpha\in ]0,1[,$ we also know
$$
AV@R_{\alpha}(-X) 
= 
\min_{x\in\R}\ex\left[\frac{(X + x)^{+}}{\alpha} - x\right] 
= 
\ex\left[\frac{(X + F_{-X}^{\rightarrow}(\alpha))^{+}}{\alpha} - F_{-X}^{\rightarrow}(\alpha)\right], 
$$
where $F_{-X}^{\rightarrow}$ stands for the right-continuous quantile function of the distribution function $F_{-X}$ of $-X$ (cf. \cite[Proposition 3.2]{acerbi2002coherence} along with \cite[Lemma A.22]{FoellmerSchied2011}). By 
$F_{-X}^{\rightarrow}(\alpha) = - F_{X}^{\leftarrow}(1-\alpha),$ we obtain
\begin{eqnarray}
\label{zweiteDarstellung} \notag
&&
AV@R_{\alpha}(-X)\\ 
\quad &=& 
\min_{x\in\R}\ex\left[\frac{(X + x)^{+}}{\alpha} - x\right] 
= 
\ex\left[\frac{(X - F_{X}^{\leftarrow}(1-\alpha))^{+}}{\alpha} + F_{X}^{\leftarrow}(1 -\alpha)\right]. 
\end{eqnarray}
Now define the mapping 
$$
\varphi:~]0,1]\times \R\rightarrow \R,~(\alpha,x)\mapsto \eins_{]0,1[}(\alpha)\ex\left[\frac{(X + x)^{+}}{\alpha} - x\right].
$$
$\varphi(\alpha,\cdot)$ is convex, and therefore continuous for any $\alpha\in ]0,1],$ and 
$\varphi(\cdot,x)$ is $\cB(]0,1])-$measurable for every $x\in\R.$ Then 
$$
\int \min_{x\in\R} \varphi(\alpha,x)~\mu_{g}(d\alpha) = \inf_{Z\in L^{1}(\mu_{g})} \int 
\varphi(\alpha,Z(\alpha))~\mu_{g}(d\alpha)
$$
(see \cite[Theorem 14.60]{rockafellar1998variational} along with \cite[Example 14.29]{rockafellar1998variational}). Moreover, the mapping $\eins_{]0,1[}F^{\leftarrow}(1 - \cdot)$ is nonnegative as well as $\cB(]0,1])-$measurable, and by  \eqref{ersteDarstellung} along with 
\eqref{zweiteDarstellung} it satisfies
\begin{eqnarray*}
\infty 
&>& 
\int \eins_{]0,1[}(\alpha)\ex\left[\frac{(X - F_{X}^{\leftarrow}(1-\alpha))^{+}}{\alpha} + F_{X}^{\leftarrow}(1 -\alpha)\right]~\mu_{g}(d\alpha)\\ 
&\geq& 
\int \eins_{]0,1[}(\alpha)F_{X}^{\leftarrow}(1 -\alpha)~\mu_{g}(d\alpha).
\end{eqnarray*}
Hence $\eins_{]0,1[}F^{\leftarrow}(1 - \cdot)\in L^{1}_{+}(\mu_{g}),$ and by \eqref{zweiteDarstellung} 
\begin{eqnarray*}
\int \min_{x\in\R} \varphi(\alpha,x)~\mu_{g}(d\alpha) 
&=& 
\inf_{Z\in L^{1}(\mu_{g})} \int 
\varphi(\alpha,Z(\alpha))~\mu_{g}(d\alpha)\\
&=& 
\int 
\varphi(\alpha,-F^{\leftarrow}(1 - \alpha))~\mu_{g}(d\alpha)\\ 
&=& 
\inf_{Z\in L^{1}_{+}(\mu_{g})\atop Z(1) = 0} \int 
\varphi(\alpha,-Z(\alpha))~\mu_{g}(d\alpha).
\end{eqnarray*}
Drawing on \eqref{ersteDarstellung} and \eqref{zweiteDarstellung} again, the statement of Lemma \ref{optimizedcertaintyequivalent} follows from Tonelli's theorem.
\end{proof}
\begin{lemma}
\label{Stetigkeitsaussage}
Let the distortion function $g$ be continous and concave. For any $\pr-$essentially bounded mapping $f$ and every nonnegative random variable $X,$ define the mapping
$$
H_{f,X}: L^{1}(\mu_{g})\rightarrow [0,\infty],~Z\mapsto \ex\left[f\cdot\int\left[\frac{(X - |Z(\alpha)|)^{+}}{\alpha} + Z(\alpha)\right]~\mu_{g}(d\alpha)\right].
$$
If $X$ is $\pr-$integrable, and if $g' (0) \doteq \lim_{\alpha\to 0+} g'(\alpha) <\infty,$ then $H_{f,X}$ is real-valued and weakly continuous w.r.t. the $L^{1}-$norm on $L^{1}(\mu_{g}).$
\end{lemma}
\begin{proof}
By \cite[Lemma 4.69]{FoellmerSchied2011} we have $g'(\beta) = \int\eins_{]\beta,1]}(\alpha)/\alpha~\mu_{g}(d\alpha)$ for every $\beta\in ]0,1[$ so that $\int 1/\alpha~
\mu_{g}(d\alpha) < \infty.$ Furthermore by integrablity of $X,$ we obtain for every $Z\in L^{1}(\mu_{g})$ and any $\tau\in\cT$
$$
\ex\left[\int\frac{(X - |Z(\alpha)|)^{+}}{\alpha}~\mu_{g}(d\alpha)\right]
\leq 
\ex\left[\int\frac{X}{\alpha}~\mu_{g}(d\alpha)\right] 
< \infty.
$$
Then by H\"older's inequality, the 
random variables $f\cdot \int\frac{(X - |Z(\alpha)|)^{+}}{\alpha}~\mu_{g}(d\alpha)$ and $f\cdot \int Z~d\mu_{g}$ are $\pr-$integrable. This implies $H(Z) < \infty.$  
\medskip

Moreover, the set $B \doteq \{Z\in L^{1}(\mu_{g})\mid \int Z~d\mu_{g}\in ]-1,1[\}$ is a weakly open neighbourhood of $0,$ and we may observe for any $Z\in B$
\begin{eqnarray*}
H_{f,X}(Z)
&\leq& 
\ex\left[f\cdot \int\frac{X}{\alpha}~\mu_{g}(d\alpha)\right] + \ex\left[f\right]\cdot\int Z~d\mu\\ 
&\leq& 
\ex\left[f\cdot \int\frac{X}{\alpha}~\mu_{g}(d\alpha)\right] + \ex\left[|f|\right].
\end{eqnarray*}
This means that $H_{f,X}$ is bounded above on the weakly open set $B.$ Therefore, as a real-valued convex mapping, it is weakly continuous (cf. \cite[Theorem 5.43]{aliprantisinfinite}). The proof is complete.
\end{proof}

\section{Appendix}
\label{AppendixA}
Let $(\oOmega,\ocF,(\ocF_{i})_{i\in\{1,\dots,m\}},\oP)$ be a filtered probability space, and let us denote by $L^{\infty}(\oOmega,\ocF_{i},\oP|_{\ocF_{i}})$ the space of 
$\oP|_{\ocF_{i}}-$essentially bounded random variables, whereas $L^{1}(\oOmega,\ocF_{i},\oP|_{\ocF_{i}})$ stands for the space of $\oP|_{\ocF_{i}}-$integrable random variables. 

We endow the product space 
$\Timesim L^{\infty}(\oOmega,\ocF_{i},\oP|_{\ocF_{i}})$ with the product topology 
$\Timesim\sigma(L^{\infty}_{i},L^{1}_{i})$ of the weak$\star$ topologies 
$\sigma(L^{\infty}_{i},L^{1}_{i})$ on $L^{\infty}(\oOmega,\ocF_{i},\oP|_{\ocF_{i}})$ 
(for $i=1,\dots,m$).
\begin{proposition}
\label{angelic}
Let $L^{1}(\oOmega,\ocF_{i},\oP|_{\ocF_{i}})$ be separable w.r.t. the weak topology $\sigma(L^{1}_{i},L^{\infty}_{i})$ for $i\in\{1,\dots,m\},$ and let $\cA\subseteq\Timesim L^{\infty}(\oOmega,\ocF_{i},\oP|_{\ocF_{i}})$ be relatively compact w.r.t. $\Timesim\sigma(L^{\infty}_{i},L^{1}_{i}).$\par 
Then for any $X$ from the 
$\Timesim\sigma(L^{\infty}_{i},L^{1}_{i})-$closure of $\cA,$ we may find a sequence $(X_{n})_{n\in\N}$ in $\cA$ which converges to $X$ w.r.t. the $\Timesim\sigma(L^{\infty}_{i},L^{1}_{i}).$
\end{proposition}
\begin{proof}
c.f. proof of Proposition B.1 in \cite{BelomestnyKraetschmer2014}.
\end{proof}

\section{Appendix}
\label{AppendixC}
Let for $m\in\N$ denote by $(\oOmega,\ocF,(\ocF_{i})_{i\in\{1,\dots,m\}},\oP)$ a filtered probability space, and let $L^{\infty}(\oOmega,\ocF_{i},\oP|_{\ocF_{i}})$ denote the space of 
$\oP|_{\ocF_{i}}-$essentially bounded random variables, whereas $L^{1}(\oOmega,\ocF_{i},\oP|_{\ocF_{i}})$ stands for the space of $\oP|_{\ocF_{i}}-$integrable random variables.

Furthermore, let the set $\overline{\cP}_{m}$ gather all 
$(A_{1},\dots,A_{m})$ from $\Timesim\ocF_{i}$ satisfying $\oP(A_{i}\cap A_{j}) = 0$ for 
$i\not= j$ and $\oP(\bigcup_{i=1}^{m} A_{i}) = 1.$ We shall endow respectively the product spaces 
$\Timesikm L^{\infty}(\oOmega,\ocF_{i},\oP|_{\ocF_{i}})$ with the product topologies 
$\Timesikm\sigma(L^{\infty}_{i},L^{1}_{i})$ of the weak$\star$ topologies 
$\sigma(L^{\infty}_{i},L^{1}_{i})$ on $L^{\infty}(\oOmega,\ocF_{i},\oP|_{\ocF_{i}})$ 
(for $k\in\{1,\dots,m\}$ and $i=k,\dots,m$). Fixing $k\in\{1,\dots,m\}$ and 
nonnegative $h\in L^{\infty}(\oOmega,\ocF_{k},\oP|_{\ocF_{k}}),$ the subset 
$\ocPinfty_{mk}(h)\subseteq \Timesikm L^{\infty}(\oOmega,\ocF_{i},\oP|_{\ocF_{i}})$ is defined to consist of all $(f_{k},\dots,f_{m})\in \Timesikm L^{\infty}(\oOmega,\ocF_{i},\oP|_{\ocF_{i}})$ fulfilling $f_{i}\geq 0$ $\oP-$a.s. for any $i\in\{k,\dots,m\}$ and 
$\sum_{i=k}^{m}f_{i} = h$ $\oP-$a.s.. For abbreviation we shall use notation 
$\ocPinfty_{m} \doteq \ocPinfty_{m1}(1).$
\begin{lemma}
\label{BanachAlaoglu}
$\ocPinfty_{mk}(h)$ is a compact subset of $\Timesikm L^{\infty}(\oOmega,\ocF_{i},\oP|_{\ocF_{i}})$ w.r.t. the topology $\Timesikm\sigma(L^{\infty}_{i},L^{1}_{i})$ for 
$k\in\{1,\dots,m\}$ and arbitrary nonnegative $h\in L^{\infty}(\oOmega,\ocF_{k},\oP|_{\ocF_{k}}).$
\end{lemma}
\begin{proof}
\,The statement of Lemma \ref{BanachAlaoglu} is obvious in view of the Banach-Alaoglu theorem.
\end{proof}

\begin{proposition}
\label{Dichtheit}
If $(\oOmega,\ocF_{i},\oP|_{\ocF_{i}})$ is atomless for every $i\in\{1,\dots,m\},$ then $\ocPinfty_{m}$ is the $\Timesim\sigma(L^{\infty}_{i},L^{1}_{i})-$closure of
$$
\{(\eins_{A_{1}},\dots,\eins_{A_{m}})\mid (A_{1},\dots,A_{m})\in\overline{\cP}_{m}\}.
$$
\end{proposition}
\begin{proof}
\,cf. Corollary C.4 in \cite{BelomestnyKraetschmer2014}.

\end{proof}
\section{Appendix}
\label{AppendixD}
Throughout the section we shall fix any probability measure $\mu$ on $\cB(]0,1]),$ denoting the usual Borel $\sigma-$algebra on $]0,1].$ Furthmore, $L^{1}_{+}(\mu)$ stands for the space of all nonnegative $\mu-$integrable random variables.
\begin{lemma}
\label{ApproximationStetigkeit}
For every $Z\in L^{1}_{+}(\mu)$ with $Z(1) = 0,$ there exists some 
sequence $(\phi_{n})_{n\in\N}$ of uniformly continuous functions on $]0,1]$ such that 
$\phi_{n}(1) = 0$ for $n\in\N,$ and $\int |\phi_{n} - Z|~d\mu\to 0.$
\end{lemma}
\begin{proof}
Let $[0,1]$ be endowed with the usual Borel $\sigma-$algebra $\cB([0,1]),$ and let 
$$
\widehat{\mu}: \cB([0,1])\rightarrow [0,1],~A\mapsto \mu(A\cap ]0,1]).
$$
This mapping is a Radon probability measure on $\cB([0,1]),$ i.e. for any $\varepsilon  > 0$ and for every $A\in\cB([0,1]),$ there is some compact subset $C$ of $[0,1]$ such 
that $C\subseteq A$ and $|\widehat{\mu}(A) - \widehat{\mu}(C)| < \varepsilon.$
\par

Now, let us fix any nonnegative, $\mu-$integrable random variable $Z.$ It may be extended to $[0,1]$ to a nonnegative, $\widehat{\mu}-$integrable random variable $\widehat{Z}$ with $\int Z~d\mu = \int\widehat{Z}~d\widehat{\mu}$ by
$$
\widehat{Z}(\alpha) \doteq 
\bcswitch
Z(\alpha)& \alpha\in ]0,1]\\
0& \alpha = 0
\ecswitch.
$$
Let $\varepsilon > 0$ be arbitrary. Since $\widehat{\mu}$ is a Radon probability measure, and since $\widehat{Z}$ is nonnegative with $Z(1) = 0,$  there exist pairwise disjoint compact subsets $C_{1},\dots,C_{r}$ of $[0,1[$ and $\lambda_{1},\dots,\lambda_{r} > 0$ such that 
$$
\int \left|~\widehat{Z} - \sum_{i=1}^{r}\lambda_{i}\eins_{C_{i}}~\right|~d\widehat{\mu} < \frac{\varepsilon}{2}.
$$
Moreover, $\sum_{i=1}^{r}\lambda_{i}\eins_{C_{i}}$ is a nonnegative, bounded upper semicontinuous mapping on $[0,1].$ Then $\tilde{h}_{n}\searrow \sum_{i=1}^{r}\lambda_{i}\eins_{C_{i}}$ for some sequence $(\tilde{h}_{n})_{n\in\N}$ of continuous mappings on $[0,1].$ Moreover, $C \doteq \bigcup_{i=1}^{r} C_{i}$ is a compact subset of $[0,1[$ so that we may find by Urysohn's lemma some continuous mapping $\varphi:[0,1]\to [0,1]$ with $\varphi(1) = 0$ and $\varphi(\alpha) = 1$ for $\alpha\in C$. Hence $h_{n} \doteq \tilde{h}_{n}\cdot\varphi$ defines an antitone sequence $(h_{n})_{n\in\N}$ of uniformly continuous mappings on $[0,1]$ such that $h_{n}\searrow \sum_{i=1}^{r}\lambda_{i}\eins_{C_{i}},$ and $h_{n}(1) = 0$ for $n\in\N$.

As a continuous mapping on a compact space $h_{1}$ is also bounded, and thus $\widehat{\mu}-$integrable. Then by dominated convergence theorem
$$
\int \left|~h_{n_{0}} - \sum_{i=1}^{r}\lambda_{i}\eins_{C_{i}}~\right|~d\widehat{\mu} < \frac{\varepsilon}{2}\quad\mbox{for some}~n_{0}\in\N.
$$
Choosing $\phi_{\varepsilon} \doteq h_{n_{0}}|_{]0,1[},$ we may conclude
$$
\int |Z - \phi_{\varepsilon}|~d\mu = \int |\widehat{Z} - h_{n_{0}}|~d\widehat{\mu} < \varepsilon.
$$
This completes the proof because $h_{n_{0}}$ is uniformly continuous.
\end{proof}

\begin{lemma}
\label{denseBernstein}
Let $\cZ_{B}$ consist of all mappings $\sum_{i=0}^{n-1}b_{i} B_{i,n}|_{]0,1]}$ with $n\in\N$ and $b_{0},\dots,b_{n-1}\geq 0,$ where $B_{i,n}$ is defined as in \eqref{BernsteinMononom}. Then $\cZ_{B}$ is a dense subset of $\{Z\in L^{1}_{+}(\mu)\mid Z(1) = 0\}$ w.r.t. the $L^{1}-$norm generated by $\mu$. 
\end{lemma}
\begin{proof}
Let $Z\in L^{1}_{+}(\mu)$ with $Z(1) = 0,$ and let $\varepsilon > 0$ be arbitrary. By Lemma \ref{ApproximationStetigkeit}, we may find some real-valued continuous mapping $\phi$ on $[0,1]$ satisfying $\phi(1) = 0,$ and 
$$
\int |\phi|_{]0,1]} - Z|~d\mu < \frac{\varepsilon}{2}.
$$
Moreover, it is well known that 
$$
\sup_{\alpha\in [0,1]}\left|~\phi(\alpha) - \sum_{i=0}^{n_{0}}\phi\left(\frac{i}{n_{0}}\right)B_{i,n}(\alpha)~\right| < \frac{\varepsilon}{2}\quad\mbox{for some}~n_{0}\in\N
$$
(cf. e.g. \cite[proof of Satz B3.1]{witting1995mathematische}). In particular
$$
\int \left|~\sum_{i = 0}^{n_{0}}\phi\left(\frac{i}{n_{0}}\right) B_{i,n}(\alpha) - Z~\right|~d\mu_{g} < \varepsilon.
$$
Finally, 
$\phi\left(\frac{0}{n_{0}}\right),\dots,\phi\left(\frac{n_{0} - 1}{n_{0}}\right)\geq 0$ and $\phi(1) = 0,$ which means that the mapping $\sum_{i=0}^{n}\phi\left(\frac{i}{n_{0}}\right) B_{i,n_{0}}|_{]0,1]}$ belongs to $\cZ_{B}.$ This completes the proof. 
\end{proof}



%
%
%

\section*{Acknowledgments.}


\bibliography{QHDbib-1}{} 
\bibliographystyle{abbrv}


\end{document}